\documentclass[leqno]{siamltex1213}

\usepackage{graphicx}
\usepackage{amssymb,amsmath,amsfonts,float}
\usepackage{thmtools,thm-restate}

\usepackage[inline]{enumitem}
\usepackage{tikz}

\newtheorem{remark}[theorem]{Remark}

\newcommand{\Lie}{\mathrm{Lie}}

\newcommand{\de}{\delta}
\newcommand{\Dt}{\Delta t}

\newcommand{\Pb}{Q_{\Delta t}} 
\newcommand{\Pbl}{Q_{\Delta t}^{\mathrm{Lie}}} 
\newcommand{\Pbs}{Q_{\Delta t}^{\mathrm{Strang}}} 

\newcommand{\mb}{\mu_{Q}} 
\newcommand{\mbl}{\mu_{\Lie}}

\newcommand{\mo}{\mu_{P}} 

\newcommand{\Po}{P_{\Delta t}} 
\newcommand{\si}{\sigma}

\newcommand{\diam}{\mathrm{diam}}
\newcommand{\arctanh}{\mathrm{atanh}}

\newcommand{\Ppath}{P_{0:T}}
\newcommand{\Qpath}{Q_{0:T}}

\newcommand{\Hpp}{H_{\mathrm{pp}}}

\newcommand{\des}{\delta_{\sigma'}(\sigma)}

\newcommand{\figpercent}{0.65}

\usepackage{color}
\newcommand{\firstRef}[1]{\textcolor{black}{#1}}
\newcommand{\secondRef}[1]{\textcolor{black}{#1}}
\newcommand{\ouredit}[1]{\textcolor{black}{#1}}

\usepackage{hyperref}

\title{Information metrics for long-time errors in splitting schemes for stochastic dynamics and parallel KMC.\thanks{The authors are partially supported by NSF (DMS-1515712, DMS-1109316)}}
\author{Konstantinos Gourgoulias\footnotemark[2], Markos A. Katsoulakis\footnotemark[2], Luc Rey-Bellet\footnotemark[2]}

\begin{document}

\maketitle

\renewcommand{\thefootnote}{\fnsymbol{footnote}}

\footnotetext[2]{Department of Mathematics and Statistics, University of Massachusetts, Amherst, MA, USA. (gourgoul@math.umass.edu)
}

\renewcommand{\thefootnote}{\arabic{footnote}}

\begin{abstract}
We propose an information-theoretic approach to analyze the long-time behavior of numerical splitting schemes for stochastic dynamics,  focusing primarily on Parallel Kinetic Monte Carlo (KMC)  algorithms. 	
Established methods for  numerical operator splittings  provide error estimates in finite-time regimes, in terms of the order of the local error and the associated commutator. 
Path-space information-theoretic tools such as the relative entropy rate (RER)  allow us  to control long-time error through  commutator calculations. 
Furthermore, they give rise to an {\em a posteriori} representation of the error which can thus be tracked  in the course of a simulation. 
Another outcome of our analysis is the derivation of   a path-space  information criterion for comparison (and possibly design) of numerical  schemes, in analogy to classical information criteria for model selection and discrimination. In the context of Parallel KMC, our analysis allows us to select schemes with improved numerical error and more efficient processor communication. We expect that such a path-space  information perspective on numerical methods will be broadly applicable in stochastic dynamics, both for the finite and the long-time regime.
\end{abstract}



\section{Introduction}
\label{sec:Intro}
Recently, schemes that depend on operator splitting have found wide applicability within the domain of simulation of complex chemical reaction systems, biological systems or those that can be modelled by appropriate Markov processes, for example interacting particle systems. The recipe of splitting the system into components that can be simulated separately in an appropriate manner has led to more efficient algorithms, sometimes because some of the components can be solved explicitly, as in chemical reaction systems~\cite{Janhke}, and others because the splitting allows for parallel computations~\cite{Arampatzis:2012, spparks}.

In parallel with the development of those algorithms, there has also been a growing amount of work towards the numerical analysis of splitting methods for stochastic dynamics in different contexts~\cite{Janhke,Arampatzis:2012, Arampatzis:2014, Petzold,Engblom, Bayati}. In particular, for the case of Parallel Lattice Kinetic Monte Carlo (PL-KMC),  the authors in~\cite{Arampatzis:2012} developed a general framework, based on semigroup theory, that connects lattice decompositions to operator splitting. Then, in~\cite{Arampatzis:2014}, error estimates were provided for bounded time intervals along with comparisons between different splitting schemes. One of the important contributions of the work was to highlight the connection of the error with the commutator associated with the splitting and how it affects the efficiency of the scheme.

Although classical techniques in numerical analysis, such as the study of the local error of the splitting scheme and expansions of the global error~\cite{Talay-Tubaro}, work well in providing error estimates for bounded intervals, the information they provide is not of great use when the focus is on long-time results. Given that a common goal is sampling from a stationary distribution and convergence occurs for large simulation times, it thus makes sense to develop methodologies for the study of long-time errors. Approaches to tackling this problem are varied. For instance, in the case of SDEs, study of the long-time behavior has been done by employing Poisson equations~\cite{MST10}. For Lie-Trotter splittings, backward error analysis~\cite{Abdulle} has been used to study the performance of the schemes in capturing the stationary distribution when simulating Langevin dynamics (but see also~\cite{Leimkuhler}). 

The main idea in this work is information-theoretical in nature, following similar successful approaches studying the irreversibility of numerical schemes~\cite{irreversibility_KPR}, sensitivity analysis~\cite{sensitivity_PK}, and quantifying the loss of information in coarse-graining of particle systems~\cite{Eva-coarse-grain}. In those, the authors use the relative entropy, along with other quantities derived from it, to both generate insights and provide computable quantities that are useful during a simulation. Besides that, approaching the problem from information theory still allows one to infer results about more classical metrics of error. For instance, one can derive upper bounds for the weak error of specific observables through the use of variational inequalities~\cite{Dupuis}.  

Our goal is to use another derived quantity, the relative entropy on path space per unit time, or relative entropy rate (RER), to quantify the long-time loss of information when using a splitting scheme. For our comparison, we fix a time step $\Dt$ and then compare the $\Dt$-skeleton chain arising from the exact process with the discrete chain we get from the approximate process. Through rigorous asymptotics, \ouredit{we provide an {\em a posteriori} error expansion of RER in terms of $\Dt$} and connect RER with quantities central to the classical analysis of splitting schemes, like the commutator and the order of the local error of the splitting method. After deriving \ouredit{computable estimators from our {\em a posteriori} expansions} for the highest-order term coefficients, we estimate them with the use of SPPARKS~\cite{spparks}, a parallel Kinetic Monte Carlo simulator, and use them to compare two well-known splitting schemes, the Lie and Strang splittings. Also, we illustrate how a practitioner can use the RER as an information criterion for selecting schemes that takes into account both long-time accuracy and communication cost. We then proceed to link the connectivity of the exact process with the RER asymptotics, which in turn  allows for greater generality in the study of different operator splittings. 

The plan for the following sections is as follows. \ouredit{In Section \ref{sec:Background}, we provide the necessary background for KMC, PL-KMC, construction and analysis of operator splitting schemes. Section~\ref{ssec:info_theory_concepts} introduces the pathwise relative entropy and relative entropy per unit time, which are the principal tools used in this work. In Section~\ref{sec:Long_time} we discuss the use of the relative entropy rate as a metric for studying the long-time loss of information that operator splitting schemes can have and motivate the use of asymptotic expansions for its study. Section~\ref{sec:IPS-PKMC} is particularly important, as we study schemes through the RER in the context of stochastic particle systems and continue to Section~\ref{sec:ising_example} with some discussion about time-step selection and the balance between error and communication in parallel KMC.  Then, in Section~\ref{sec:info-crit}, we highlight some connections between the proposed framework and model selection with information criteria. Section~\ref{sect:conn} studies the RER for operator splitting schemes in a more general setting with the use of ideas from graph theory. Finally, in Section~\ref{sec:transient}, we demonstrate that the RER can also be applied in transient regimes, before the simulation has converged to stationarity.}

\section{Background}
\label{sec:Background}
\ouredit{Consider that the stochastic process of interest is} an ergodic Continuous Time Markov Chain (CTMC) $X_t$ on a finite, but possibly still significantly large, state space $S$. This stochastic process can be completely defined by its \textit{transition rates}, $q(\si,\si')$, which describe the probability of an update from state $\si$ to state $\si'$ in an infinitesimal period of time. \ouredit{That is,}
\begin{align}
P(X_{t+\Dt}=\si'|X_t=\si)=P_{\Dt}(\si,\si')=q(\si,\si')\Dt+o(\Dt), \si\neq \si'.
\label{eq:infinitesimal_description}
\end{align}
Kinetic Monte Carlo (KMC) works by simulating the embedded Markov Chain $Y_n=X_{t_n}$, with jump times $t_n, t_n\sim \exp(\lambda)$. The parameter $\lambda(\si)$ is the total rate when the system is at state $\si$, 
\begin{align}
    \lambda(\si)=\sum_{\underset{\si'\in S}{\si'\neq \si}}q(\si,\si').
    \label{eq:total_rate}
\end{align}
This allows us to write the transition probabilities of the embedded Markov Chain  $p(\si,\si')=q(\si,\si')/\lambda(\si)$. \ouredit{We can also define the infinitesimal generator $L$ that corresponds to the Markov chain as follows. First, consider $f$: bounded and continuous function} on the state space $S$. Then, $L$ acts on $f$ at the state $\si$ as
\begin{align}
    L[f](\si)=\sum_{\si'\in S}q(\si,\si')\left(f(\si')-f(\si)\right).
    \label{eq:infinit_gen}
\end{align}
Note that $L[\de_{\si'}](\si)=q(\si,\si')$ for all states $\si,\si'$,\firstRef{where $\de_{\si'}(\si)=\de(\si,\si')$ is a Dirac probability measure. We shall also use the notation $L^k$ for the resulting operator after $k$ successive compositions of $L$. Because $L^k[\de_{\si'}](\si)=L^{k-1}[L[\de_{\si'}]](\si)$, we see that, for any $k$,  $L^k[\de_{\si'}](\si)$ is a computable object that depends on the transition rates.}

\ouredit{Under fairly general conditions~\cite{Kipnis-Landim}, the transition probability of the Markov process can be written as in semigroup form, i.e.\ $P_t(\si,\si')=e^{Lt}\delta_{\si'}(\si)$}. \ouredit{In the case of interest to us, $L$ is going to be a bounded operator and such operators allow for a representation of the semigroup with a series expansion.}

\begin{lemma}
\label{lem:semigroup_expansion}
Let $L$ be a linear \& bounded operator, $L:C_{b}(S)\to C_{b}(S)$, with $C_{b}(S)$ being the set of continuous and bounded functions on the space $S$.  Then $L$ generates a uniformly continuous semigroup $e^{tL}$ which we can express in power series form.
\begin{align}
e^{tL}&=\sum_{k=0}^{\infty}\frac{t^k}{k!}L^{k}.
              \label{eq:semigroup_expansion_formal}
\end{align} 
\end{lemma} 
\begin{proof}
This is a classical result for which many references exist, see for example chapter 1, page 2 of Pazy A.  \cite{Pazy:1983}.
\end{proof}

Thus, making use of Lemma \ref{lem:semigroup_expansion}, we can write the transition probabiliy as
\begin{align}
P_{t}(\si,\si')&=e^{tL}\de_{\si'}(\si)=\sum_{k=0}^{\infty}\frac{t^k}{k!}L^{k}[\de_\si'](\si),\ \si,\si'\in S.
              \label{eq:semigroup_expansion}
\end{align} 

\subsection{Constructing approximations by semigroup splitting}
\label{sect:FSKMC}

We will now give the foundations of approximations by splitting methods, as applied to the simulation of CTMCs and proceed with how those ideas are applied in the case of parallel  lattice KMC.

As mentioned earlier, the transition probability of the CTMC of interest can be written as  $e^{tL}\de_{\si'}(\si)$. The goal is for us to design a splitting scheme that can approximate the action of $e^{tL}$. In our context, this leads to a new CTMC. One way to build such a scheme is to start with a splitting of the infinitesimal generator $L$ \eqref{eq:infinit_gen} into \ouredit{components $L_1, L_2$ with $L=L_1+L_2$. Then, if we consider a positive $T$} and by using the Trotter product formula~\cite{Trotter:1959},  we have
\begin{align}
e^{TL}=\lim_{n\to \infty}(e^{T/nL_1}e^{T/nL_2})^n.
\label{eq:trotter}
\end{align}
Correspondingly, if we now fix $n\in\mathbb{N}$ and set $\Dt=T/n$, we can write approximations of $e^{TL}$ by using \eqref{eq:trotter}. \ouredit{For example, two such approximations are:}
\begin{equation}
    \begin{aligned}
        e^{TL}&\simeq \left (e^{\Dt L_1} e^{\Dt L_2}\right )^{n},\text{ (Lie),}\\
        e^{TL}&\simeq \left (e^{\Dt/2 L_1} e^{\Dt L_2}e^{\Dt/2 L_1}\right )^{n}, \text{ (Strang)}.
    \end{aligned}
    \label{eq:two_main_split}
\end{equation}
Therefore for a one step transition from $t=0$ to $\Dt$, \eqref{eq:two_main_split} can be written as
\begin{equation}
\begin{aligned}
e^{L\Dt}&\simeq e^{\Dt L_1}e^{\Dt L_2},\\
e^{L\Dt}&\simeq e^{\Dt/2L_1}e^{\Dt L_2}e^{\Dt/2L_1}.
\end{aligned}
\label{eq:one_step_split}
\end{equation}

\firstRef{Operator splittings can also be carried out with multiple components, such as $L=L_1+L_2+L_3+L_4$. Such a splitting is used for 2D lattice decompositions in SPPARKS~\cite{spparks}. All arguments can be simply extended to those cases, but we stick to two components, $L_1,L_2$, for notational convenience.}

Throughout this work, we use $\Po(\si,\si')$ to denote the probability $e^{L\Dt}\de_{\si'}(\si)$ and $\Pb(\si,\si')$ for the approximations arising from splittings of the semigroup.  Since $L$ is a bounded operator, we can \ouredit{express $\Po$ as expansion  \eqref{eq:semigroup_expansion}}. If we pick $L_1,L_2$  so that they are also bounded, then we can express $\Pb$ as an expansion too. For example, for the Lie splitting  
\begin{align}
\exp(\Dt L_1)\exp(\Dt L_2)\de_\si'(\si)
&=\sum_{k=0}^{\infty}\frac{\Dt^k}{k!}\left(k!\cdot \sum_{m=0}^{k}\frac{L_1^m}{m!}\cdot \frac{L_2^{k-m}}{(k-m)!}\right)\de_{\si'}(\si),
\label{eq:lie_exp}
\end{align}
\ouredit{which can be showed by multiplying the semigroup expansions of $\exp(\Dt L_1)$ and $\exp(\Dt L_2)$.} Thus, if we use the notation:
\begin{align}
\label{eq:L_QofLie}
L_Q^k:&=k!\cdot \sum_{m=0}^{k}\frac{L_1^m}{m!}\cdot \frac{L_2^{k-m}}{(k-m)!}
\end{align}
we can write~\eqref{eq:lie_exp} in the form
\begin{align}
\Pb(\si,\si')=\sum_{k=0}^{\infty}\frac{\Dt^k}{k!}L^k_{Q}[\de_{\si'}](\si).
\label{eq:gen_split_power_exp}
\end{align}
\ouredit{By the definition of $L_Q^k$ in Equation~\eqref{eq:L_QofLie},} $L_Q^0=I$, $L_Q^1=L$, $L_Q^2=(L_1^2+L_2^2+2L_1L_2)$, and so on, \ouredit{for the case of the Lie splitting}. By a similar argument, we can write an expansion like~\eqref{eq:gen_split_power_exp} for other \ouredit{operator splitting approximations}. \firstRef{In general, $L_Q$ is not a generator of a Markov process and, in that case, $L^k_Q$ is not equal $L_Q$ after $k$ compositions but is defined in the context of the expansion in~\eqref{eq:gen_split_power_exp}. The slight abuse of notation allows us to compare the expansion of the exact process~\eqref{eq:semigroup_expansion} with expansions of the approximating schemes of the form~\eqref{eq:gen_split_power_exp}.} 

One way to compare the accuracy of using $\Pb$ as opposed to $\Po$ is to calculate the local error between expansion~\eqref{eq:semigroup_expansion} and~\eqref{eq:gen_split_power_exp}. As an example, here are the corresponding relations for the Lie and Strang splittings. We use $\Pbl, \Pbs$ \ouredit{for Lie and Strang respectively}. We will also use the notation $[L_1,L_2]:=L_1L_2-L_2L_1$ \ouredit{to denote the operator that} captures the failure of $L_1$ and $L_2$ to commute. \ouredit{By using the expansions}~\eqref{eq:semigroup_expansion},~\eqref{eq:gen_split_power_exp}, we can show that 
\begin{align}
 \Po(\si,\si') &=\Pbl(\si, \si') +\frac{1}{2}[L_1,L_2]\de_{\si'}(\si)\Dt^2+O(\Dt^3),\label{eq:comm-lemma-Lie}\\
 \Po(\si,\si') &=\Pbs(\si,\si')+\frac{1}{24}\left([L_1,[L_1,L_2]]-2[L_2,[L_2,L_1]]\right)\de_{\si'}(\si)\Dt^3 \label{eq:comm-lemma-Strang}\\
 &+O(\Dt^4).\nonumber
\end{align}
From \ouredit{Relations~\eqref{eq:comm-lemma-Lie} and~\eqref{eq:comm-lemma-Strang}}, we observe that the Strang splitting has a better local error compared to  Lie \ouredit{($\Dt^3$ versus $\Dt^2$)}. \firstRef{Therefore, if we prescribe an error tolerance, the Strang scheme will be able to accommodate a larger $\Dt$ than the Lie scheme. With a larger $\Dt$, we will be able to take larger steps with the same tolerance during the simulation, and this is  especially important for Parallel KMC, as we strive for balance between error accumulation and efficiency.}

To be able to discuss more general \ouredit{operator splitting} approximations to $\Po$, we introduce the following helpful lemma.
\begin{lemma}[Local order of error \& commutator]
\label{lem:local_order_of_error}
Let  $\Po(\si,\si')=e^{L\Dt}\delta_{\si'}(\si)$ and $\Pb(\si,\si')$ an approximation of $\Po$ via a splitting scheme. Then, there is a function $C:S\times S\to \mathbb{R}$ and an integer $p$, $p>1$, such that

\begin{align}
\Po(\si,\si') = \Pb(\si,\si')+C(\si,\si')\Dt^p + o(\Dt^p).
\label{eq:prop:local_order_of_error}
\end{align}
\end{lemma}
We will refer to $C(\si,\si')=(L^p-L_Q^p)\de_{\si'}(\si)$ as the \textit{commutator} and to $p$ as the \textit{order of the local error}.
\begin{proof}
The result is immediate by using representations \eqref{eq:semigroup_expansion}, \eqref{eq:gen_split_power_exp}, since for $\si,\si'\in S$,
\begin{align*}
\Po(\si,\si') - \Pb(\si,\si')=\sum_{k=0}^{\infty}\frac{\Dt^k}{k!}\left(L^k-L_{Q}^k\right)[\de_\si'](\si).
\end{align*}
Then, $p$ is the smallest non-negative integer such that $L^p\neq L_Q^p$.  \ouredit{This of course implies that $L^k=L_Q^k$ for $k<p$.}
\end{proof}

 \firstRef{Equations~\eqref{eq:comm-lemma-Lie} and~\eqref{eq:comm-lemma-Strang} are examples of this lemma for the cases of the Lie and Strang splittings respectively. Although in the case of Lie we were able to write the form of $L_Q^k$ explicitly for all $k$ (Equation~\eqref{eq:L_QofLie}), this is not a requirement and we only need to know $L_Q^p$ to compute the commutator and that object arises naturally when subtracting the two expansions, \eqref{eq:semigroup_expansion} and \eqref{eq:gen_split_power_exp}.}

\begin{remark}
Relation \eqref{eq:prop:local_order_of_error} is central to the numerical analysis of splitting schemes, as it is the starting point to the derivation of upper bounds for the local and global error \cite{Arampatzis:2012,Arampatzis:2014,Petzold}. \firstRef{Even though our focus in  this manuscript is on operator splitting schemes for parallel KMC, as long as an expression for the local error such as Equation~\eqref{eq:prop:local_order_of_error} exists, a similar analysis can be carried out for other types of schemes.}
\end{remark}

As we will see in the follow-up, the commutator has many desired properties. Since it is equal to $(L^p-L^p_Q)\de_{\si'}[\si]$, \ouredit{ and both $L^p[\de_{\si'}](\si)$ and $L^p_Q[\de_{\si'}](\si)$  depend on the known transition rates $q$, the commutator is a} \textit{computable} object for every pair of states $(\si,\si')$.
\ouredit{We will see in Section~\ref{sec:a-posteriori-RER} that for parallel KMC the work required in order to compute the commutator can scale appropriately with the system size.}

\subsection{Parallel Lattice KMC and splitting schemes}
\label{sect:ParallelLatticeKMC}
We consider the case of PL-KMC as an application of the ideas in the previous section concerning approximations by semigroup splitting. Further discussion on the ideas of this section can be found in Arampatzis et al. \cite{Arampatzis:2012, Arampatzis:2014}.

Our main motivating example for PL-KMC is an interacting particle system. Let $\Lambda\subset \mathbb{Z}^d$ be a square lattice with $N$ sites. At each site of it, $x\in \Lambda$, we define an order parameter $\si(x)\in \Sigma=\{0,1,\ldots,K\}$. This parameter can be, for example, the species that occupies the lattice site $x$. For instance, in the Ising model, $\si(x)=0$ would imply that the lattice site $x$ is empty and $\si(x)=1$ that a particle occupies $x$. The CTMC of interest is $\{\sigma_t\}_{t\geq 0}$, $\si_t=\{\si_t(x):x\in\Lambda \}$, with state space $S=\Sigma^{\Lambda}$. At every $t$, $\si_t$ represents a snapshot of the different occupancies of the lattice. We can describe the dynamics of such a system by looking at the individual spin changes \ouredit{at different lattice sites}. Two more properties that are common among such systems and which we will also assume is that the transitions between states of $\sigma_t$ are \textit{localized} and that \ouredit{they} only involve a finite number of lattice sites per transition step. Localization implies that the probability that a certain transition will happen (the order parameter of a finite collection of lattice sites will change) only depends on the values of $\sigma$ on a neighborhood around those lattice sites. In other words, transitions depend on local (neighborhood) rather than global (whole lattice) information (see Figure \ref{fig:lattice_decomp}). 

We can formalize \ouredit{localization} by looking at the implication for the transition rates of the process $\si_t$. Following the notation introduced in \cite{Arampatzis:2012}, let us assume that at time $t$, $\si_t=\sigma$. Now, \ouredit{we can express the transition rate for a jump} to a new state $\si^{x,\omega}$ as

\begin{align}
q(\si,\si^{x,\omega})=q(x,\omega;\sigma)
\label{eq:trans_rates}
\end{align}
where $x\in \Lambda$ and $\omega$ is an index of the set of all possible configurations, $S_x$, that correspond to an update at a lattice neighborhood $\Omega_x$ of the site $x$. When the only allowed transition is spin-flipping, that is, starting with $\si$, we can only go to states $\si'$ that differ in the order parameter of one lattice site $x$, we will write $\si'$ as $\si^x$ to denote the resulting state after the transition. It follows that for $\sigma_t$ we have an infinitesimal generator:
\begin{align}
    L[f](\si)=\sum_{x\in \Lambda}\sum_{\omega\in S_x}q(x,\omega;\sigma)\left(f(\si^{x,\omega})-f(\si)\right).
    \label{eq:back:gen}
\end{align}

We can simulate the process $\sigma_t$ via standard KMC, as described in the beginning of Section~\ref{sec:Background}. Then the system would progress in time steps $t_n\sim \exp(\lambda(\si))$, where $\lambda(\si)$ is the total rate when the system is at state $\si$, as defined in~\eqref{eq:total_rate}. Since the total rate scales with the size of the lattice and the magnitude of the transition rates, a large or highly reactive model would be simulated slowly by classical KMC. The goal then, as realized in \cite{Arampatzis:2012}, is for a fixed $\Dt>0$ to design an approximation to the exact process $e^{\Dt L}$ via a splitting method in such a way that allows for asynchronous computations.

 To begin, we note that any decomposition of the lattice into non-overlapping sub-lattices $\Lambda_i$ also induces a decomposition of the generator \eqref{eq:back:gen}, that is 
\begin{align}
    L[f](\si)=\sum_{i=1}^{n}\sum_{x\in\Lambda_i}\sum_{\omega\in
        S_x}q(x,\omega;\sigma)\left(f(\si^{x,\omega})-f(\si)\right).
    \label{eq:back:split_gen}
\end{align}
Due to the \ouredit{localization} of the system, we can decompose  the lattice $\Lambda$ into  \ouredit{$n$ sub-lattices, $\Lambda_{i}$}, so that transitions in some sub-lattices are independent from transitions in others, see Figure~\ref{fig:lattice_decomp}. With two groups, $G_1=\{\Lambda_i:i \text{ even}\}, G_2=\{\Lambda_i:i \text{ odd}\}$, we can split $L$ into 
\begin{equation}
    \begin{aligned}
        L_j[f](\si)&:=\sum_{x\in G_j}\sum_{\omega\in
            S_x}q(x,\omega;\sigma)\left(f(\si^{x,\omega})-f(\si)\right),\ j=1,2,\\
            L[f](\si)&=L_1[f](\si)+L_2[f](\si).
    \end{aligned}
            \label{eq:splitting_to_groups}
\end{equation}

\begin{figure}[h]
\centering

					\begin{tikzpicture}[scale=0.14]
										
										\fill [fill=green!50] (0,40) rectangle (10,30);
										\fill [fill=white] (1,39) rectangle (9,31);
										
										\foreach \x in {0,1}{
											\foreach \y in {2}{
												\fill [fill=red!50] (\x*10+10*\x,\y*10) rectangle (10*\x+10+10*\x,10*\y+10);
											}
										}
										
										\foreach \x in {0,1}{
											\foreach \y in {3}{
												\fill [red!50] (\x*10+10*\x+10,\y*10) rectangle (10*\x+10+10*\x+10,10*\y+10);
										   }
										}
										
										\def\color{blue!}
										\fill [\color!40,ultra thick,fill=\color] (4,35) rectangle (6,34);
										\fill [\color!40,ultra thick,fill=\color] (5,34) rectangle (6,33);
										\fill [\color!40,ultra thick,fill=\color] (6,33) rectangle (5,36);
										\fill [\color!40,ultra thick,fill=\color] (7,35) rectangle (6,34);
										\fill [black ,ultra thick, step=1] (0,20) grid (40,40);
										\fill [black, ultra thick,step=10] (0,30) grid (40,40);
										\end{tikzpicture}
\caption{A checkerboard decomposition of a 2D lattice. Red sub-lattices correspond to group $G_1$ and white ones to $G_2$. For comparison, a nearest neighborhood region (n.n. region) is also shown (solid black cross). Transitions involving the center of that region only depend on the state of its nearest neighbors. So, if we pick the sub-lattices much larger than the size of an n.n. region, transitions in different sub-lattices belonging to the same group are independent. A site $x$ is said to belong to the boundary of its sub-lattice if part of its n.n. region is outside that sub-lattice (the green region is the collection of all such points for the first sub-lattice). If a transition occurs at such a site $x$, then an update needs to be made to the boundary information of all other sub-lattices for which $x$ belongs to a n.n. region. }
\label{fig:lattice_decomp}
\end{figure}
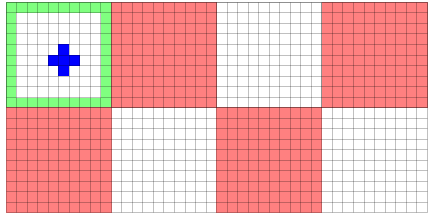

Thus, by the formulas in \eqref{eq:splitting_to_groups},  we can use the ideas of the previous section to construct splitting approximations to $e^{L\Dt}$. Those \ouredit{can also be interpreted} as computation schedules for the parallel algorithm. Such schedules set two attributes of the simulation: (a) in what order to simulate the two groups asynchronously and (b) for how much time to simulate each group per \ouredit{time-step} (which the user controls with the $\Dt$ parameter). A demonstration of how PL-KMC works is shown in Figure~\ref{fig:pkmc}.
 \begin{figure}[h]
\centering
\includegraphics[width=\figpercent\textwidth]{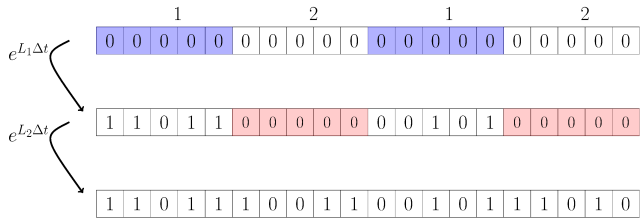}
\caption{One step of PL-KMC in the $1D$ case, where all of the spin values are set to zero initially while using the Lie splitting. After the lattice is decomposed into non-overlapping sub-lattices, here blue (indexed as $1$) and red (indexed as $2$), the algorithm proceeds by first simulating all blue sub-lattices independently by standard KMC until a time $t=\Dt$ is reached for all of them. Once that is done, the lattices in the second group are simulated in the same way. This results to the process $\si_t$ on the whole lattice being propagated forward in time by $\Dt$. In-between the simulation of each group, communication between the processes is required in order to correct for the mismatch on the boundaries of the sub-lattices. The resulting error due to the mismatch is controlled by the commutator $C$~\cite{Arampatzis:2014}.}
\label{fig:pkmc}
\end{figure}

In general, the larger the $\Dt$, the less different processes need to communicate to resolve inconsistencies during a run. This is a fact for any simulation algorithm that can be expressed in the above operation splitting framework, e.g. SPPARKS and others~\cite{Arampatzis:2012}. Since communication is the usual bottleneck of PL-KMC algorithms, a practitioner would like to pick $\Dt$ as large as possible, given a fixed tolerance. One of the important insights of the analysis in \cite{Arampatzis:2014} is that the commutator controls this relationship. Simply put, a small $C(\cdot, \cdot)$ (as defined in Lemma~\ref{lem:local_order_of_error}) allows for a larger step size $\Dt$. 

\section{Information metrics for comparing dynamics at long times} 
\label{ssec:info_theory_concepts}
We will now introduce the main tools from information theory. In later sections, our focus will be to compare the exact process, $X_t$, and an approximation of it, $Y_t$, via their $\Dt$-skeleton sub-processes. That is, given a fixed $\Dt>0$ and $M\in \mathbb{N}$, we look at the discrete-time Markov processes $X_{n\Dt}$ and $Y_{n\Dt}$ for $n\in \{0,\ldots,M\Dt\}, T=M\Dt$. For this reason, we now introduce those concepts for discrete-time processes.

Let us assume two discrete-time Markov processes $X_{n}$ and $Y_{n}$ on a countable state space $S$  with transition probabilities $P$ and $Q$ respectively. We also assume that for each process exists a corresponding unique stationary distribution $\mo$ and $\mb$. Assuming $X_{0}$ ($Y_{0}$) is distributed according to $\mo$ ($\mb$), we can then calculate the probability of a specific path for each process. For example, if we fix \ouredit{a positive integer $M, T=M\Dt$}, and pick an $\vec{x}\in S^M$, then we have
\begin{align*}
P_{0:T}(\vec{x})&=P(X_{T}=x_M,\ldots, X_{0}=x_0)=\mo(x_0)P(x_0,x_1)\cdots P(x_{M-1},x_{M}).
\end{align*}
Similarly, by changing $P$ to $Q$, we can calculate the path probability for $Y_n$. 

Assuming one would prefer a path of length $T$ of the process $Y_{n}$ to infer results about a same length path of $X_{n}$, how much information about $X_{n}$ would be lost by such a method? This is a central question in coding theory and one way to quantify the information loss is through the idea of relative entropy (RE),
\begin{align}
\label{eq:RE}
R(Q_{0:T}|P_{0:T}):=\sum_{\vec{x}\in S^M}Q_{0:T}(\vec{x})\log\frac{Q_{0:T}(\vec{x})}{P_{0:T}(\vec{x})}
\end{align}
Our definition here is with respect to the path measures $P_{0:T}, Q_{0:T}$, but we can apply the relative entropy to more general probability measures too. For this object to be properly defined, we need to have that $Q_{0:T}$ is absolutely continuous with respect to $P_{0:T}$, that is $P_{0:T}(\vec{x})=0$ implies $Q_{0:T}(\vec{x})=0$. Other important properties of the relative entropy rate are the following : \ouredit{
\begin{enumerate*}
\item  $R(Q_{0:T}|P_{0:T})\geq 0$ for any $Q_{0:T}, P_{0:T}$ (Gibbs' inequality),
\item $R(Q_{0:T}|P_{0:T})=0\Leftrightarrow P_{0:T}=Q_{0:T}$.
\end{enumerate*}}
Note though that the relative entropy does not qualify as a metric in the classical sense, as it is not symmetric and does not satisfy the triangle inequality. It can however still be thought of as a distance between distributions and is useful as a building block for other information measures. For a more complete exposition on relative entropy and its properties, see Cover and Thomas~\cite{Cover-Thomas}.

\ouredit{Although the pathwise RE is a suitable quantity to measure the similarity of the two path-measures, it is computationally demanding to calculate, especially in the case of parallel KMC, where we do not have $Q_{0:T}$ and $P_{0:T}$ explicitly. For this reason, we look at a related object, the relative entropy per unit time, or relative entropy rate (RER).} Given a probability measure $\nu_0$, $\nu_0(\vec{x})=\nu_0(x_0), \vec{x}\in S^{T}$, \ouredit{the RER with} respect to $\nu_0$ \ouredit{is defined as:}  
\begin{align}
H_{\nu_0}(Q|P):=\sum_{\vec{x}\in S^M} {\nu_0}(\vec{x})Q(x_0,x_1)\log\frac{Q(x_0,x_1)}{P(x_0,x_1)}.
\end{align}
Given another measure $\mu_0$, we can use the chain rule for the relative entropy~\cite{Cover-Thomas} to relate RE and RER as
\begin{align}
R(Q_{0:T}|P_{0:T})&=R(\mu_0|\nu_0)+\sum_{i=1}^{M}H_{\nu_i}(Q|P),\label{eq:gen_re_rer_rel}\\
\nu_k(x_0,\ldots,x_{k-1})&=\nu_{0}(x_0)\prod_{m=1}^{k-1}Q(x_{m-1},x_{m}).\nonumber
\end{align}
In particular,  when sampling from the stationary distribution corresponding to $Q$, that is $\nu_0=\mb$, then $H_{\nu_i}=H_{\mb}=H$ for all $i$. Then, 
\begin{align}
H(Q|P)&=\sum_{x_0,x_1\in S}\mb(x_0)Q(x_0,x_1)\log\frac{Q(x_0,x_1)}{P(x_0,x_1)}.\label{eq:rer}
\end{align}
This also simplifies Equation~\eqref{eq:gen_re_rer_rel} to

\begin{align}
R(Q_{0:T}|P_{0:T})&=M\cdot H(Q|P) + R(\mb|\mo).\label{eq:rel_entropy_split_PK}
\end{align}
In~\eqref{eq:rel_entropy_split_PK}, $R(\mb|\mo)$ is the relative entropy of $\mb$ with respect to $\mo$, capturing the loss of information between the exact and approximate stationary distribution. Note that $R(\mb|\mo)$ does not depend on the length of the path. Instead, the term that quantifies the dependence on $T$ is $H(Q|P)$. Therefore, any difference between the two stationary measures becomes negligible for large times, \ouredit{which is a first advantage to studying the pathwise RE through the simpler RER.}

\subsection{Information metrics and  observables}
Further justification for the fact that the RER is the right quantity to track can be given by considering time-averaged observables. For instance, if $f$ is a function of the state space, then such an observable would be
$$M\cdot F_{M}(\{X_n:n=0,\ldots,M-1\})=\sum_{k=0}^{M-1}f(X_k).$$
An important performance metric for the approximation is the weak error:
\begin{align}
|\mathbb{E}_{P[0,T]}[F_M]-\mathbb{E}_{Q[0,T]}[F_M]|,\ T=M\Dt.
\label{eq:weak_error_ta}
\end{align}
In recent work~\cite{Dupuis}, uncertainty quantification (UQ) bounds have been developed for the weak error that are of the form:
\begin{equation}
\begin{aligned}
\Xi_{-}(Q_{[0,T]}\| P_{[0,T]};M\cdot F_M)/M&\leq \mathbb{E}_{P[0,T]}[F_M]-\mathbb{E}_{Q[0,T]}[F_M]\\
&\leq \Xi_{+}(Q_{[0,T]}\| P_{[0,T]};M\cdot F_M)/M. 
\end{aligned}
\label{eq:uq_bounds_fin}
\end{equation}
The quantities $\Xi_{\pm}(Q_{[0,T]}\| P_{[0,T]};M\cdot F_M)$ are defined as goal-oriented divergences~\cite{Dupuis}, taking into account the observable $F$, and such that $\Xi_{\pm}(Q_{[0,T]}\| P_{[0,T]};M\cdot F_M)=0$, if $Q_{[0,T]},=P_{[0,T]}$ or $f$ is deterministic. \firstRef{Note that the bound in~\eqref{eq:uq_bounds_fin} is robust, see Theorem 3.4 in \cite{CD:13}, as well as  \cite{jie-scalable-info-ineq}: if we consider a positive $\eta$ and all $\Pb$ such that $R(\Pb|\Po)<\eta$, then the upper bound in~\eqref{eq:uq_bounds_fin} is attained.}

\firstRef{ Dividing~\eqref{eq:uq_bounds_fin} by $M$  and letting $M$ go to infinity gives an inequality with respect to the stationary measures $\mb,\mo$ of the scheme, $\Pb$, and the exact process, $\Po$, respectively:
\begin{align}
\label{eq:weak-error-limit}
\xi_{-}(\Pb\|\Po;f)\leq \mathbb{E}_{\mb}[f]-\mathbb{E}_{\mo}[f]\leq \xi_{+}(\Pb\|\Po;f),
\end{align}
where $\xi_{\pm}(\Pb\|\Po;f)=\lim_{M\to\infty}\Xi_{\pm}(\Qpath\|\Ppath;F)/M$. But $\xi_{\pm}$ also admit a variational representation as
\begin{equation}
\label{eq:xi-variational}
\begin{aligned}
\xi_{+}(\Pb\|\Po;f)&=\inf_{c\geq 0}\left\{\frac{1}{c}[ \lambda_{\Pb,\Po}(c)+H(\Pb\|\Po) ]\right\},\\
\xi_{-}(\Pb\|\Po;f)&=\sup_{c\geq 0}\left\{-\frac{1}{c}[ \lambda_{\Pb,\Po}(-c)+H(\Pb\|\Po) ]\right\},
\end{aligned}
\end{equation}
with $\lambda_{\Pb,\Po}(c)$ in \eqref{eq:xi-variational} to be the logarithm of the maximum eigenvalue of the matrix with entries $\Po(x,y)\exp(c\cdot (f(y)-\mathbb{E}_{\mo}[f]))$ (see~\cite{jie-scalable-info-ineq} for details). Especially when $H(\Pb|\Po)$ is small and through the asymptotic expansion of $\xi_{\pm}$, an upper bound for the weak error at stationarity can be given (following the ideas in~\cite{Dupuis, jie-scalable-info-ineq}):
\begin{align}
|\mathbb{E}_{\mb}[f]-\mathbb{E}_{\mo}[f]|&\leq \sqrt{\upsilon_{\mo}(f)}\sqrt{2H(\Pb|\Po)}+O(H(\Pb|\Po))\label{eq:linearized-ineq},\\
\upsilon_{\mo}(f)&=\sum_{k=-\infty}^{\infty}\mathbb{E}_{\mo}[f(X_k)f(X_0)].\label{eq:auto-correlation}
\end{align}
}

\firstRef{
Inequality~\eqref{eq:linearized-ineq} connects the long-time loss of accuracy that the weak error captures with the relative entropy rate and $\upsilon_{\mo}(f)$, which is the integrated auto-correlation function for the observable $f$ and a quantity we can estimate during the simulation. As a consequence of~\eqref{eq:linearized-ineq}, any further results on the asymptotic behavior of $H(\Pb|\Po)$ with respect to $\Dt$ can be simply translated to the weak error point-of-view. 
}

\section{Long-time error behavior of splitting schemes}
\label{sec:Long_time}
In this section, we compare the RER between two different processes. One of them will always be the $\Dt$-skeleton process derived from the CTMC we wish to simulate, with transition probability
\begin{align}
\Po(\si,\si')=e^{L\Dt}\de_{\si'}(\si).
\label{eq:exact_proc}
\end{align}
\ouredit{This exact $\Dt$-process will be compared with the} $\Dt$-skeleton process derived from an \ouredit{operator} splitting of \eqref{eq:exact_proc}. \ouredit{Such approximations will be denoted with $\Pb$.}
\firstRef{We note here that the discretization \eqref{eq:exact_proc}
of the original Markov process with semigroup $e^{tL}$ with respect to $\Dt$ is only
carried out  as a means  to compare the original process with the
approximations $\Pb$. The transition kernel $\Po$ is just a particular instance
of the transition matrix  of the continuous Markov process with semigroup $P_t=e^{tL}$, so there is no approximation error
 in \eqref{eq:exact_proc}.
In fact, using the $\Dt$-skeleton corresponds to sub-sampling from the CTMC at every $\Dt$.
}

Our goal is to show the dependence of the RER to various quantities of interest that are usually computed for short-time error analysis. We will see that the commutator, the order of the local error, and other quantities, make an appearance in the asymptotic results we develop. We limit our discussion to the case that $\Dt$ is in $(0,1]$, as this is the interval where splitting schemes are most accurate. We also assume throughout this section that $L$ is a bounded operator. We will often refer to the splittings previously discussed, Lie and Strang, which define discrete processes with transition probabilities
\begin{equation}
\begin{aligned}
	\Pbl(\si,\si')&=e^{L_1\Dt}e^{L_2\Dt}\delta_{\si'}(\sigma),\\
	\Pbs(\si,\si')&=e^{L_1\Dt/2}e^{L_2\Dt}e^{L_1\Dt/2}\delta_{\si'}(\sigma).
\end{aligned}
\label{eq:splittings}
\end{equation}
Here $L$ is the original generator and $L=L_1+L_2$ with $L_1, L_2$ \ouredit{assumed  bounded as operators}. \ouredit{For instance, in the case of parallel KMC, $L_1,L_2$ will be imposed by the domain decomposition of the lattice, see Figure~\ref{fig:lattice_decomp}.}

Before we move on to the analysis, we need to address a last issue. As mentioned before, our main tool will be asymptotic expansions of the RER with respect to $\Dt$. We will then use those to do comparisons for different $\Dt$, so it is important to first account for the scaling of RER with \ouredit{respect to} that parameter. The situation can be best illustrated by the worst case scenario, when the order of the local error \firstRef{between  two Markov semigroups, $\Pb^A,\ \Pb^B$, is equal to one. }
\begin{lemma}
\label{prop:rer_scaling}
Let $L_A,L_B$ be bounded generators of Markov Processes, $L_A\neq L_B$, \firstRef{with corresponding transition probabilities $\Pb^A,\Pb^B$ }. Then, 
$$
H(\Pb^B|\Pb^A)=O(\Dt).
$$
\end{lemma}
\begin{proof}
\ouredit{Proof follows the ideas in Theorem~\ref{th:spparks_th}. The argument is provided in the supplementary material.} 
\end{proof}

\begin{remark}
Using Lemma~\ref{prop:rer_scaling}, we can readily see that given \ouredit{an operator} splitting scheme \firstRef{$\Pb$ that approximates the exact $\Po$, we expect a scaling at least of the type
$H(\Pb|\Po)=O(\Dt)$. To correct for the $\Dt$ scaling}, we will instead work with a $\Dt$-normalized RER. \firstRef{That is, we re-define the RER as:}
\begin{align}
\label{eq:RER-normalized}
H(\Pb|\Po)\firstRef{:=}\frac{1}{\Dt}\sum_{\si,\si'}\mb(\si)\Pb(\si,\si')\log\left(\frac{\Pb(\si,\si')}{\Po(\si,\si')}\right).
\end{align}
\end{remark}

\ouredit{We wish to use the RER (Equation~\eqref{eq:RER-normalized}) to study the long-time loss of information between $\Pb$ and $\Po$. However, in the case of Parallel KMC, those are difficult to calculate explicitly, hence we turn to asymptotic expansions instead. We will see that the terms in those expansions depend on the transition rates and, under suitable ergodic assumptions, can be estimated during the simulation. }

\section{RER analysis for Parallel KMC}
\label{sec:IPS-PKMC}
We will now study an example from a class of interacting particle systems, limiting our discussion to the Lie and Strang splittings. Given two states $\si,\si' \in S$ and $x$ lattice site, $\si(x)\in \{0,1\}$, we have that the transition rates $q$ are
\begin{align}
q(\si,\si')=\begin{cases}
q(\si,\si^x)>0,& \si'=\si^x,\\
0,& \text{else.}
\end{cases}
\label{eq:rate_assumption}
\end{align}
The rates in Equation~\eqref{eq:rate_assumption} provide a particular example of an adsorption/desorption system. Other mechanisms can be incorporated into~\eqref{eq:rate_assumption} such as diffusion, reactions with multiple components or with particles that have many degrees of freedom~\cite{Arampatzis:2012}.

Given a lattice $\Lambda$ with $N$ sites, we are interested in simulating the process $\si_t=\{\si_t(x):x\in \Lambda\}$ in parallel with an \ouredit{operator} splitting method, so we apply the ideas in Section~\ref{sect:ParallelLatticeKMC} to that end. We first decompose the lattice into non-overlapping sub-lattices \ouredit{(see Figure~\ref{fig:lattice_decomp})} and this induces a decomposition of the generator into new generators $L_1, L_2$ as in~\eqref{eq:splitting_to_groups}. Then, for any $T>0$, the \ouredit{adsorption/desorption} system can be simulated in $[0,T]$ using the \ouredit{parallel KMC algorithm}. From the short-time error analysis, we can control the error by computing the commutator, $C(\cdot,\cdot)$, and the order of the local error that corresponds to the \ouredit{operator splitting scheme} we use. For example, we know that for the Lie splitting that order is $p=2$ and $C(\si,\si')=[L_1,L_2]\de_{\si'}(\si)/2$ (see Lemma~\ref{lem:local_order_of_error} and \ouredit{Equation~\eqref{eq:comm-lemma-Lie}}). By using the properties of the generators $L_1, L_2$ along with our assumption in \eqref{eq:rate_assumption}, we can show that 
\begin{equation}
\begin{aligned}
C(\si, \si')=\ouredit{[L_1,L_2]\des/2}= &\frac{1}{2}\sum_{x,y\in \Lambda} f_1(x,y;\si)\delta_{\si'}(\si^{x,y})- f_2(x,y;\si)\delta_{\si'}(\si^x)\\
&-\frac{1}{2}\sum_{x,y\in \Lambda} f_3(x,y;\si)\delta_{\si'}(\si^y),
\end{aligned}
\label{eq:comm_Lie}
\end{equation}
where $f_1,f_2$ and $f_3$ only depend on the transition rates $q$. We remind here that $\si^{x,y}$ stands for the resulting state $\si'$ after a spin-flip of an initial state $\si$ at lattice sites $x,y, x\neq y$. A full description of the above formula along with proof can be found in the supplementary material. 

\begin{remark}
\label{rem:comm_prop}
 Formula \eqref{eq:comm_Lie} for the Lie commutator has two important properties. First, it is computable for any pair $(\si,\si')\in S\times S$ \firstRef{as it only depends on the transition rates $q$}. Second, it is surely equal to zero if $\si'\neq \si^{x,y}$ and $\si'\neq \si^x$ for all $x,y\in \Lambda, x\neq y$, due to the $\de_{\si'}$ appearing in the different sums. \firstRef{We will also see that the sum in \eqref{eq:comm_Lie} needs only to be evaluated for the neighboring lattice sites $x,y$ that are not both in the same group. For instance, in Figure~\ref{fig:lattice_decomp}, we would only need to evaluate the sum over the green boundary regions of every sub-lattice, which makes the computation of the commutator much simpler (see Remark~\ref{rem:scal_and_computation} for a complexity analysis).  Those properties hold for commutators of other operator splitting schemes too, see~\cite{Arampatzis:2014} and Section~\ref{sect:conn}. }
\end{remark}

\ouredit{To study the asymptotic behavior of the RER, we will need to quantify the dependence of various combinations of $\Po$ and $\Pb$ to $\Dt$. To this end, we use the following facts, both of which stem from Lemma~\ref{lem:local_order_of_error}.
\begin{align}
	\Po(\si,\si')-\Pb(\si,\si')&=C(\si,\si')\Dt^p+o(\Dt^p)\label{eq:PminusQ},\\
	\Po(\si,\si')+\Pb(\si,\si')&=2\de_{\si'}(\si)+2q(\si,\si')\Dt+o(\Dt)\label{eq:PplusQ}\\
	&=2\Pb(\si,\si')+C(\si,\si')\Dt^p+o(\Dt^p)\label{eq:PplusQ2}.
\end{align}}
We are now able to write an asymptotic result for RER for \ouredit{the Lie and Strang operator splittings in parallel KMC under the assumption in Relation~\eqref{eq:rate_assumption}.}
\begin{theorem}
\label{th:spparks_th}
Let $\Dt\in (0,1)$ and $\si_{n\Dt}$ on the lattice $\Lambda$ with transition probability $\Po(\si,\si')=e^{L\Dt}\delta_{\si'}(\si)$ for $\si,\si'\in S$. Then, let $L_1+L_2$ be a splitting of $L$ based on a decomposition of the lattice $\Lambda$. Assuming that property \eqref{eq:rate_assumption} holds for the rates, then if there exists a state $\si\in S$ and lattice sites \ouredit{distinct} $x,y$ such that the Lie commutator $C(\si, \si^{x,y})\neq 0$, we have that
\begin{align}
H(\Pbl|\Po)=O(\Dt^1) \text{ (Lie)}.
\end{align}
Similarly, if there exists a state $\si\in S$ and \ouredit{distinct} lattice sites $x,y,z$ such that $C(\si, \si^{x,y,z})\neq 0$, 
\begin{align}
H(\Pbs|\Po)=O(\Dt^2) \text{ (Strang)}. 
\end{align}
\end{theorem} 

\begin{proof}
\ouredit{We will first show the result for the Lie case and then note the differences in the proof for the Strang case. Thus, we denote $\Pbl$ by $\Pb$, $\mbl$ by $\mb$, and consider a $\Dt\in (0,1)$.  As we wish to construct an asymptotic expansion for the RER (Equation~\eqref{eq:RER-normalized}), we first need to expand the logarithm. Given a positive $x$} and by the definition of $tanh^{-1}$,
\begin{align}
\log(x)=2 \arctanh\left(\frac{x-1}{x+1}\right)=2\sum_{k=0}^{\infty}\frac{1}{2k+1}\left(\frac{x-1}{x+1}\right)^{2k+1}.
\label{eq:i_arctanh_def}
\end{align}
This expansion of the logarithm converges for every $x>0$, as can be seen by applying the root convergence test. \ouredit{Thus, expanding the logarithm part of the RER, we get:} 
\begin{align}
\Dt\cdot H(\Pb|\Po)=& -2\sum_{\si,\si'}\mu_{Q}(\si)\Pb(\si,\si')\frac{\Po(\si,\si')-\Pb(\si,\si')}{\Pb(\si,\si')+\Po(\si,\si')}\label{eq:RER-expand}\\
&+2\sum_{\si,\si'}\mu_{Q}(\si)J(\Dt;\si,\si'),\nonumber\\
J(\Dt;\si,\si')&:=\Pb(\si,\si')\sum_{k=1}^{\infty}\frac{1}{2k+1}\left(\frac{\Pb(\si,\si')-\Po(\si,\si')}{\Pb(\si,\si')+\Po(\si,\si')}\right)^{2k+1} .\label{eq:naive:F}
\end{align}

 \ouredit{We will study the asymptotic behavior of both parts of the RER in Equation \eqref{eq:RER-expand}. First, applying Equation~\eqref{eq:PplusQ} to the denominator of the fraction in  \eqref{eq:RER-expand} and carrying out the simplifications, we have}
\begin{equation}
	\begin{aligned}
		\Dt\cdot H(\Pb|\Po)=&-2\sum_{\si,\si'}\mb(\si)\left(\Po(\si,\si')-\Pb(\si,\si')+G(\Dt;\si,\si')\right)\\
		&+2\sum_{\si,\si'}\mb(\si)J(\Dt;\si,\si')\label{eq:proof4}.
	\end{aligned}
\end{equation} 
Now, \ouredit{since $\Pb,\Po$ are transition probabilities, $\sum_{\si'\in S}\Po(\si,\si')-\Pb(\si,\si')=0$ for all $\si\in S$, and thus the corresponding part of Equation~\eqref{eq:proof4} is zero.} To progress, we need to study the dependence on $\Dt$ of $J,G$. \ouredit{First, for $G$ in Equation~\eqref{eq:proof4},}
\begin{align}
G(\Dt;\si, \si')=\frac{(\Po(\si,\si')-\Pb(\si,\si'))C(\si,\si')\Dt^2}{\left(2\Pb(\si,\si')+\Dt^2C(\si,\si')+o(\Dt^2)\right)}+o(\Dt^{2}).
\label{eq:proof6}
\end{align}
\ouredit{To expose the dependence of the numerator of \eqref{eq:proof6} to $\Dt$, we use \eqref{eq:PminusQ} to get}
\begin{align}
G(\Dt;\si, \si')=\frac{(C(\si,\si'))^2}{2\Pb(\si,\si')+\Dt^2C(\si,\si')+o(\Dt^2)}\Dt^{4}+o(\Dt^{2}).
\label{eq:proof7}
\end{align} 
We wish to show that $G(\Dt;\si,\si')=O(\Dt^2)$. From the explicit form of the commutator in \eqref{eq:comm_Lie} and Remark \ref{rem:comm_prop}, we can see that we only need to study $G$ in the cases that $\si'=\si^x$ or $\si'=\si^{x,y}$, given a state $\si$ and lattice sites $x,y$, since otherwise $C(\si,\si')=0$. \ouredit{Let us consider $\si'=\si^{x,y}$. Since the order of the local error is equal to two,} from expansion \eqref{eq:gen_split_power_exp} and the fact that $L_Q[\de_{\si^{x,y}}](\si)=L[\de_{\si^{x,y}}](\si)$ and $L[\de_{\si^{x,y}}]=q(\si,\si^{x,y})=0$ \ouredit{(see the property in \eqref{eq:rate_assumption})}, we have
\begin{align}
\Pb(\si,\si^{x,y})=\frac{\Dt^2}{2}L^2_{Q}[\de_{\si'}](\si)+o(\Dt^2).
\label{eq:proof8}
\end{align}
Thus, applying~\eqref{eq:proof8} \ouredit{to the denominator} of~\eqref{eq:proof7}, 
\begin{align}
G(\Dt;\si, \si^{x,y})&=\frac{(C(\si,\si^{x,y}))^2}{\Dt^2\cdot (L^2_{Q}[\de_{\si^{x,y}}](\si)+C(\si,\si^{x,y}))+o(\Dt^2)}\Dt^{4}+o(\Dt^{2})\nonumber\\
&=\frac{(C(\si,\si^{x,y}))^2}{L^2_{Q}[\de_{\si^{x,y}}](\si)+C(\si,\si^{x,y})}\Dt^{2}+o(\Dt^{2})
\label{eq:proof9}
\end{align} 
By similar calculations, we can show that $G(\si,\si^x)=O(\Dt^3)$, if $C(\si,\si^x)\neq 0$ for that $x\in \Lambda$. Regardless, this would be a lower order, since $\Dt< 1$. Thus, $G(\Dt;\si,\si')$ \ouredit{is indeed of order} $\Dt^2$. Next, we will account for $J(\Dt;\si,\si')$. If $\si'=\si^{x,y}$, then 
\begin{align}
J(\Dt;\si,\si^{x,y})=\Pb(\si,\si^{x,y})\sum_{k=1}^{\infty}\frac{1}{2k+1}\left(\frac{\Pb(\si,\si^{x,y})-\Po(\si,\si^{x,y})}{\Pb(\si,\si^{x,y})+\Po(\si,\si^{x,y})}\right)^{2k+1}.
\label{eq:proof:F_expansion}
\end{align}
\ouredit{Because $\Pb(\si,\si^{x,y})=O(\Dt^2)$ and $\Pb(\si,\si^{x,y})\pm\Po(\si,\si^{x,y})=O(\Dt^2)$, we get}
$$
J(\Dt;\si,\si^{x,y})=O(\Dt^2),
$$
since, for $\si'=\si^x$, $J(\Dt;\si,\si^x)=O(\Dt^4)$ and this is a lower order when $\Dt<1$. Therefore, $H(\Pb|\Po)=O(\Dt^1)$. Note that all of the terms of the series in \eqref{eq:proof:F_expansion} contribute a term of order $\Dt^2$, so the coefficient of $\Dt^2$ in the asymptotic expansion of the RER will be a result of the summation of all those terms.

Finally, we discuss the differences in our argument for the proof of the Strang case. First, the order of the local error for Strang is $p=3$, so every time we use formula \eqref{eq:PminusQ} in the proof, we would introduce a term of order $\Dt^3$ instead of $\Dt^2$. Then, using an expression for $C(\cdot, \cdot)$ similar to \eqref{eq:comm_Lie} but for the Strang case, we would show that 
$$J(\Dt;\si,\si^{x,y,z})=O(\Dt^3)=G(\Dt;\si,\si^{x,y,z})$$
 for $x,y,z\in \Lambda$ and $x\neq y\neq z$. This would then give the result for Strang.
\end{proof}

\subsection{Building biased a-posteriori estimators for the RER}
\label{sec:a-posteriori-RER}
\ouredit{Theorem~\ref{th:spparks_th} shows that the long-time accuracy with respect to the RER of the two operator spllitting schemes, Lie and Strang, scales with $\Dt$ in the same way the global error does.  However, it also exposes the first terms in the asymptotic expansion of the RER for Lie and Strang. Essentially, }
\begin{align}
H(\Pbl|\Po)&=A\Dt+o(\Dt)\label{eq:RER-A}\\
H(\Pbs|\Po)&=B\Dt^2+o(\Dt^2)
\end{align}
\ouredit{where $A,B$ are the corresponding highest order RER coefficients. Those have an explicit form that depends on the system one wishes to simulate and the commutator $C(\si,\si')$ corresponding to the scheme. We focus to the case of the Lie operator splitting, though similar comments can also be made for Strang. For systems with transition rates satisfying the property in~\eqref{eq:rate_assumption}, the highest-order coefficient $A$ appearing in~\eqref{eq:RER-A} has the form:
\begin{align}
\label{eq:RER-Lie-top-order}
A=\sum_{\si}\mbl(\si)\sum_{x,y\in \Lambda} C_{\Lie}(\si,\si^{x,y})F_{\Lie}(\si,\si^{x,y}),
\end{align}
where $C_{\Lie}$ is the Lie commutator (see Equation~\eqref{eq:comm-lemma-Lie}) and $F_{\Lie}$ is a quantity that depends on the splitting (see Equations~\eqref{eq:exact_a} and~\eqref{eq:exact_b} in appendix for examples on how this F can look for different splittings).} \firstRef{Both $C$ and $F$ can be expressed in terms of the transition rates of the process $q$, i.e. they are computable for any state $\si$ and $x,y\in \Lambda$. }\ouredit{Therefore, $A$ in~\eqref{eq:RER-Lie-top-order} can be estimated via an ergodic average when simulating with the Lie scheme and hence, for small $\Dt$, $H(\Pbl|\Po)\simeq A\Dt$.}

\ouredit{At first glance, computing coefficient~\eqref{eq:RER-Lie-top-order} involves work that scales with the size of the lattice}. \ouredit{However, it was shown in Lemma~5.15 of~\cite{Arampatzis:2014} that the commutator only depends on the boundary regions between sub-lattices (see Figure~\ref{fig:lattice_decomp}). We will continue this discussion in Section~\ref{sec:ising_example}, where we consider an adsorption-desorption system.} 
\ouredit{We will also see that, apart from a comparison of the schemes in terms of the long-time loss of information, the estimators of RER can also be of use in tuning parameters of the scheme ($\Dt$, domain decomposition, etc.). We will then consider the behavior of the RER when simulating other systems in Section~\ref{sect:conn}.}

\section{Error vs.\ communication and time-step selection}
\label{sec:ising_example}
In this section, we explore the balance between numerical error and processor communication in Parallel KMC, 
in the context of a specific example.
Let us assume a bounded two-dimensional lattice, $\Lambda\subset \mathbb{Z}^2$ with $100\times 100$ sites. At each site $x$, we have a spin variable, $\si(x)\in \Sigma=\{0,1\}$, with $\si(x)=0$ denoting an empty site and $\si(x)=1$ an occupied one. Our model in this case is going to be an \textit{adsorption-desorption} one, although the analysis would similarly apply for other mechanisms (diffusions, reactions, etc., see~\cite{Arampatzis:2012} for more details). The transition rates we will use correspond to spin-flip Arrhenius dynamics. Given a lattice site $x$, we may also define the nearest-neighbor set $\Omega_{x}=\{z\in \Lambda:|z-x|=1\}$. The transitions rates are then
\begin{align}
q(\si,\si^x)&=q(x,\si)=c_1(1-\si(x))+c_2\si(x)e^{-\beta U(x)},\label{eq:ising_q}\\
U(x)&=J_0\sum_{y\in \Omega_{x}}\si(y)+h,
\end{align}
where $c_1,c_2,-\beta, J_0$ and $h$ are constants that can be tuned to generate different dynamics. We remind that $\si^x$ denotes the result of a spin-flip at lattice position $x$ if we start from state $\si$. Note that the transition rates \eqref{eq:ising_q} have the property \eqref{eq:rate_assumption}. \ouredit{When considering a jump from $\si$ to $\si^x$, $q$ only depends in the spin values of the sites close to $x$ (through $U(x)$).} Since transitions are localized, we can thus employ a geometrical decomposition of the lattice, as described in Section \ref{sect:FSKMC}, and simulate the system in parallel. To accomplish this, we used Sandia Labs' SPPARKS code, a Kinetic Monte Carlo simulator~\cite{spparks}.

\ouredit{From Table~\ref{tab:comm_cost} and Remark~\ref{rem:scal_and_computation}, we can see that the cost of computing quantities that depend on the commutator scales as $O(N)$ for an $N\times N$ lattice. As the highest order coefficients of the RER also depend on the commutator (see Section~\ref{sec:a-posteriori-RER}), those also scale as $O(N)$. We can take advantage of the knowledge of the scaling by defining a per-particle RER (pp-RER)}. That is, 
\begin{align}
\label{eq:pp-RER}
\Hpp(\Pb|\Po)\ouredit{:=}\frac{1}{N}H(\Pb|\Po).
\end{align}
This way, setting a tolerance for the pp-RER will have the same meaning across different system sizes. \ouredit{We confirmed that $O(N)$ is the right scaling of the per-particle RER with respect to system size via simulation, as we saw that for increasing $N$, $\Hpp(\Pb|\Po)\simeq o(1)$.}

To estimate the top-order coefficients of the pp-RER expansion, we simulated the system until convergence to the stationary distribution was established. After that, every sample simulated by SPPARKS~\cite{spparks} was used to calculate the estimates. Note that, in this case, we show an over-estimate of $B$, so results for the Strang splitting will be even better than the ones presented in Figure~\ref{fig:comparison_2d_ising}. It is possible to get an estimator that converges to the exact value of $B$ by adding all of the positive terms in $L_S^3[\de_\si'](\si)$ to the denominator of~\eqref{eq:MB}. 
\begin{figure}[h]
\centering
\includegraphics[width=\figpercent\textwidth]{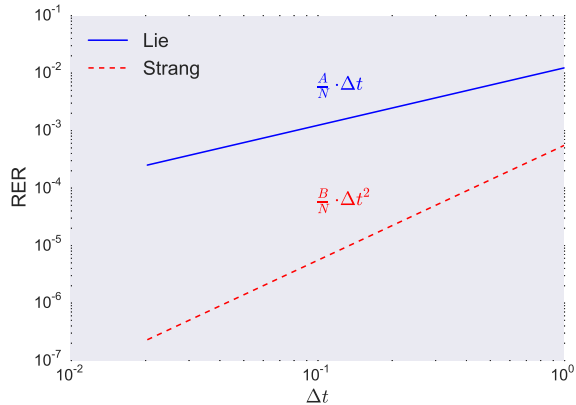}
\caption{Logarithmic scale - Comparison between $\Dt$ and the estimate of the per-particle RER for Lie \& Strang. Estimates for the constants $A,B$ come from the simulation of a 2D Ising model on a $100\times 100$ lattice with final time $T=1000$. Simulation was done in parallel with SPPARKS. }
\label{fig:comparison_2d_ising}
\end{figure}
Figure~\ref{fig:comparison_2d_ising} illustrates the difference in long-time accuracy between the two splittings. Since this is a logarithmic plot, most of the difference is made by Strang having a different order than Lie.

\begin{remark}[On the efficiency of computing \ouredit{the highest order coefficients of the expansion of the RER for the Lie and Strang operator splittings.}]
\label{rem:scal_and_computation}

In the case of a checkerboard decomposition of the lattice (see Figure \ref{fig:lattice_decomp}), we can calculate in exactly how many sites we need to evaluate the rates in order to calculate the commutator. However, for our purposes, upper bounds will be more appropriate. Table~\ref{tab:comm_cost} offers a comparison of those bounds when we decompose a $N\times N$ lattice into $m^2$ sub-lattices, assuming nearest neighbor interactions. Notice that the cost is larger for Strang due to the complexity of the corresponding commutator.
\begin{table}[H]
\centering
	\begin{tabular}{|c|c|c|}
		\hline
		& Lie& Strang\\
		\hline
		Upper bound of the commutator cost& & \\(normalized by number of sites, $N^2$)& $2(m+1)/N$& $6(m+1)/N$\\
		\hline
	\end{tabular}
	\caption{Upper bounds \ouredit{(normalized by lattice size)} on the number of lattice sites we need to evaluate the transition rates at in order to calculate the commutator for each \ouredit{operator} splitting. Assuming that a checkerboard decomposition \ouredit{into $m^2$ sub-lattices} of an $N\times N$ lattice is used, as in Figure \ref{fig:lattice_decomp}. The commutator also encodes the cost of communication between the processes. As $N$ grows, the cost of communication is smaller, as the processes spend more time simulating on the sub-lattices than updating each others boundaries.}
\label{tab:comm_cost}
\end{table}
\end{remark}

On a more practical note, a user of a splitting scheme may instead like to see the flipped relationship. That is, given a fixed tolerance, what is the maximum time window during which the simulation can run asynchronously? If we interpret tolerance as a fixed value of $\Hpp(\Pb|\Po)$ during the simulation, then the relationship with $\Dt$ is the one in Figure~\ref{fig:tol_vs_dt}.
\begin{figure}[h]
\centering
\includegraphics[width=\figpercent\textwidth]{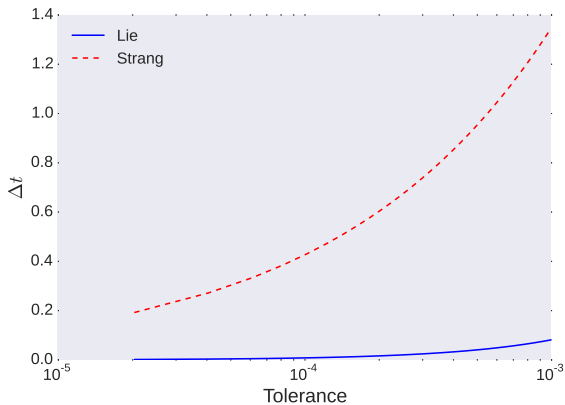}
\caption{Comparison between tolerance and $\Dt$. The difference in order of the pp-RER between the two splittings allows for a larger splitting time-step $\Dt$ given a fixed tolerance. This is similar to the behavior of the error in~\cite{Arampatzis:2014}, although the RER allows us to make this statement for $T\gg 1$.}
\label{fig:tol_vs_dt}
\end{figure}
There we can see that if our error tolerance with respect to the pp-RER is $10^{-3}$, then any $\Dt$ smaller than 0.7 works for the Strang splitting. To get within the same tolerance with Lie, $\Dt$ has to be less than $0.02$, a substantially small step-size for parallel computations. As is expected, a smaller step-size comes with larger communication cost and thus a longer computation for the same tolerance. This can be seen in Figure~\ref{fig:timing_comp}. 

\begin{remark}
Figures~\ref{fig:tol_vs_dt} and~\ref{fig:timing_comp} illustrate the very practical consequences of the theory. Interest in highly accurate splitting schemes in PL-KMC stems from a tolerance-versus-communication point of view. A user of such a scheme would like for it to be as accurate as possible, therefore the step size, $\Dt$, should be relatively small. However, for the scheme to be efficient, $\Dt$ should be large enough for every processor to have a substantial amount of work to do before communications are in order. A good balance can be reached in-between and a scheme that is more accurate allows for a larger $\Dt$ while holding the same error tolerance. Given that the RER captures long-time behavior, this is an important comparison between the schemes.
\end{remark}

\begin{figure}[h]
\centering
\includegraphics[width=\figpercent\textwidth]{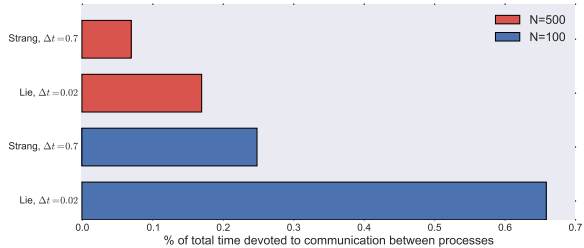}
\caption{Percentage of time each scheme devotes to communication in a fixed time interval, $[0,T]$, for a square $N\times N$ lattice when simulating an Ising-type system, using four processes and for $T=3000$. Note that for the $\Dt$ considered, the pp-RER tolerance is $10^{-3}$ for both schemes. Due to the considerably smaller step size of the Lie scheme, a larger chunk of time is devoted to communication. This is more apparent in the case of a moderately small lattice, $N=100$, where the time spent updating the other processes is over $60\%$ of total time. Communication cost is more severe when $N$ is smaller. By Remark \ref{rem:scal_and_computation}, as $N$ grows, communication should take less of the total time, as the processes spent more time simulating than updating their boundaries. }
\label{fig:timing_comp} 
\end{figure}

\subsection{The per-particle RER as an efficient diagnostic quantity for parallel KMC}
\secondRef{The discussion above about the per-particle RER, \eqref{eq:pp-RER} suggests the use of these estimates as efficient diagnostic quantities for comparing schemes. As discussed in the previous section, we can infer the scaling of the top-order coefficient of the RER by the properties of the commutator. Consequently, we can ``normalize'' the RER (as in~\eqref{eq:pp-RER}) by that scaling to derive a similarity measure that does not depend on system size. This is significant as it allows practitioners to compare schemes and tune parameters ($\Dt$, domain decomposition, etc.) on a system of smaller size and thus avoid further slowing down of the target simulation, which is crucial for complicated systems.
Overall, our approach can be viewed as  a diagnostic tool that allows to  compare different parallelization
schemes based on operator splitting.
}

\section{Some connections with Model Selection and Information Criteria}
\label{sec:info-crit}
\ouredit{The interacting particle system application considered in Section~\ref{sec:ising_example}} allows us to look at the RER via a statistical lens. The goal is to compare two models, $\Pb^1,\Pb^2$, of the actual distribution $\Po$ by utilizing simulated data. From this standpoint, our methodology is nothing more than model selection. There is an abundance of literature towards tackling the comparison of different models, given a sufficiently large amount of data. A prominent example is the use of information criteria in the Model Selection literature, like Akaike~\cite{Akaike1} and Bayesian~\cite{Bayes1}. Those provide estimates for the information lost \ouredit{compared to a given data set} by using one approximate model instead of another, without requiring knowledge of the true model. 

The approach in this work is very similar in nature.  As stated before, motivated by Theorem \ref{th:spparks_th}, we can express the RER in each case as 
\begin{align*}
H(\Pb^i|\Po)=A_i\Dt^{p_i}+o(\Dt^{p_i}), p_i\geq 1, i\in \{1,2\}.
\end{align*}
For instance, in the case of the Lie splitting, $A_1=A$ as defined in \eqref{eq:exact_a}, $p_1=2$ and for Strang $A_2=B, p_2=3$, as defined in \eqref{eq:exact_b}. Given simulated data and for a small fixed $\Dt$, we can  estimate the \ouredit{ coefficients $A_i$}. Comparison of the schemes can now be done through
\begin{align}
H(\Pb^1|\Po)-H(\Pb^2|\Po)=A_1\Dt^{p_1}-A_2\Dt^{p_2}+o\left(\Dt^{\min(p_1,p_2)}\right).
\label{eq:rer_info_crit}
\end{align}
The difference $A_1\Dt^{p_1}-A_2\Dt^{p_2}$ shares the properties of the information criteria previously mentioned while also introducing some new ones, namely: 

\begin{enumerate}
\item It is a computationally tractable quantity.  
\item Compares the schemes in terms of long-time information loss (through $p_1,p_2$). 
\item Takes into account communication cost of each scheme (through $A_1, A_2$ and associated commutators).\label{item:comm_cost}
\end{enumerate}
Thus, as an information criterion, RER differences like in Equation~\eqref{eq:rer_info_crit} offer a different perspective through which to pick a splitting scheme over another. \ouredit{A new element in our approach, compared to the earlier vast literature in Information Criteria, is the use of RER instead of the standard relative entropy. Using RER  allows us to compare stochastic dynamics models  and in a data context, correlated time series.}

\section{Generalizations, Connectivity, and Relative Entropy Rate}
\label{sect:conn}
\ouredit{
Up to this point, we have analyzed the RER with respect to the leading order in $\Dt$ for the case of a stochastic particle system (see Theorem~\ref{th:spparks_th}). In this section, we study the RER in a more general setting and illustrate that it captures more details about the system and the scheme used than one would expect. We will also see how the order of the RER can change depending on those details, resulting in some cases to schemes of higher accuracy. 
}
\begin{definition}[Restriction of a generator]
\label{def:restriction}
Let us have set $A$ with $A \subset S\times S$ and $L$ be an infinitesimal generator of a Markov process with associated transition rates $q$. Then, the restriction $L|_{A}$ of $ L $ is defined as
\begin{equation}
\label{eq:restriction_definition}
L|_{A}[f](\si)=\sum_{\si'\in S}q_A(\si,\si') \left(f(\si')-f(\si)\right),\ \si\in S, 
\end{equation}
where $q_A(\si,\si')=q(\si,\si')\cdot \chi_A(\si,\si') $, $\chi_A$ is the characteristic function of set $A$ and $f$ is a continuous and bounded function on the state space $ S $. 
\end{definition}

\ouredit{
We assume that the operator $L$ is splitted into $L_1, L_2$, and that both are \textit{restrictions} of $L$. Note that Definition~\ref{def:restriction} is general enough to include the splittings used in PL-KMC. For example, the generators $L_1,L_2$ in~\eqref{eq:splitting_to_groups} are precisely of that form, with the groups $G_i$ playing the role of the sets ``$A$''. From another point of view, restrictions respect the original process in that the transition rates that correspond to $L|_A$ are either the same as the old ones or zero. 
}

Before we can construct an asymptotic estimate for the relative entropy rate, we need to first introduce some of the tools we will use. Let $\si, \si'$ be states of a CTMC on a countable state space and let $q$ be the associated transition rates. Then, a path $\vec{z}=(z_0,\ldots, z_n)$ from $\si$ to $\si'$ is a finite sequence of \ouredit{distinct states $z_i$ such that $z_0=\si, z_n=\si'$, and $\prod_{i=0}^{n}q(z_i,z_{i+1})>0.$} The length of a path will be denoted by $|\vec{z}|=|(z_0,\ldots,z_n)|=n$ and we will use $\mathrm{Path}(\si\to \si')$ for the set of all paths from $\si$ to $\si'$. Thus, we are now able to define a distance between states by looking at the length of the shortest path that connects them.

\begin{definition}[Distance between states]
\label{def:distance}
Let $q$ be the transition rates of a Continuous Time Markov Process over a countable state space $S$. Then, let $\si,\si'\in S$, $\si\neq \si'$. The distance $d_{q}$ between the two states is defined as 
\begin{align}
d_{q}(\si,\si'):=\min\left\{|\vec{z}|:\vec{z}\in \mathrm{Path}(\si\to \si')\right\}
\label{eq:state_distance}
\end{align}
In the case that the two states are disconnected, i.e. $\mathrm{Path}(\si\to \si')=\emptyset$, then $d(\si,\si')=+\infty$. Given those distances, one can also define the diameter of the space as
$$
\diam(S)=\max_{(\si,\si')\in S\times S}\{d(\si,\si')\}.
$$
\end{definition}
This notion of distance comes from graph theory and is known as the geodesic distance. When there is no ambiguity concerning the transition rates used, we will drop the $q$ from the notation, using $d$ instead of $d_{q}$. $d$ is not a metric in the classical sense, since it does not have to be symmetric, that is $d(\si,\si')\neq d(\si',\si)$ in general. However, it satisfies the triangle inequality. In addition, the distances depend only on the transition rates, i.e. they are time independent. We will refer to those distances as the \textit{connectivity} of the state space for the Markov Chain with transition rates $q$. The importance of \ouredit{using such a distance} can be seen in the following result concerning compositions of the infinitesimal generator $L$.
\begin{lemma}
\label{le:supp_oper_pow_control}
Let $L$ be an infinitesimal generator \ouredit{of a Markov process, with corresponding transition rates $q$ and let $\si'$ be some state of the process. Then,}

\begin{align*}
\{\si:L^{n}[\delta_{\si'}](\si)\neq 0\}\subseteq \{\si:d(\si,\si')\leq n\}=B_{n}(\si').
\end{align*}
\end{lemma}
\begin{proof}
The proof is by induction. Argument can be found in supplementary material.
\end{proof}

In other words, for a fixed state $\si'$, if $d(\si,\si')>n$ then $L^n[\de_\si'](\si)=0$. The  set $B_n(\si')$ contains all states that are connected with $\si'$ with $n-2$ or less in-between states. We will also use the notation $S_n(\si'):=\{\si:d(\si,\si')=n\}$.

Since our primary interest is in studying approximations based on splitting our generator $L$ to $L_1, L_2$, it makes sense to have an extension of the previous result to compositions of $L_1, L_2 $. The following lemma is the generalization of Lemma \ref{le:supp_oper_pow_control} to compositions of restrictions. We will use the notation $L^k|_A$ to denote the $k$th composition of generator $L$ where, instead of the original transition rates, we use $q_A$. 

\begin{lemma}
\label{lem:compo_support}
Let us have the state space $S$ and $S\times S=A\cup B, A\cap B=\emptyset$, along with generators $L_1=L|_{A}$, $L_2=L|_{B}$. We fix $\si'\in S$ and $k, m\in \mathbb{N}$. Then, 
\begin{align*}
\left\{\si:L_1^{k}\left[L_2^{m}[\de_{\si'}]\right](\si)\neq 0\right\}\subseteq\{\si:d(\si,\si')\leq k+m\}.
\end{align*}
\end{lemma}

\begin{proof}
Induction argument similar to that of Lemma \ref{le:supp_oper_pow_control}, see supplementary materials.
\end{proof}

Lemma \ref{lem:compo_support} can be simply extended to more complicated compositions by the use of similar arguments. Thus, if every composition of $L_1, L_2$ is controlled in the sense of Lemma \ref{lem:compo_support}, then it is not difficult to see that the same control holds for collections of them of the same order, i.e. if we fix $\si'\in S$ and $k\in \mathbb{N}$,
\begin{align}
\{\si:L^k_Q[\de_\si'](\si)\}\subseteq \{\si:d(\si,\si')<k\}.
\label{eq:L^k_supp}
\end{align}

We can use restrictions of generators as building blocks for splitting schemes. A point often made in this work is the importance of the commutator in studying those schemes. Thus, it makes sense to have a relation between connectivity and the commutator.

\begin{lemma}[Support of the commutator]
\label{lem:comm_support}
Let $L$ be the generator of a Markov process and $L_1,L_2$ restrictions of that generator. Let also $\Dt>0$. Then, assume $\Pb$ is an approximation of $\Po$ by using a splitting scheme of order $p$ with associated commutator $C$. Then, for fixed $\si'\in S$,
$$
\{\si:C(\si,\si')\neq 0\}\subseteq \{\si:d(\si,\si')\leq p\}.
$$
\end{lemma}

\begin{proof}
In Lemma \ref{lem:local_order_of_error}, we defined the commutator as $C(\si,\si')=(L^p-L^p_Q)\de_{\si'}(\si)$. From Lemma \ref{le:supp_oper_pow_control}, we have that if $d(\si,\si')>p$, then $L^p[\de_\si'](\si)=0$ and from Equation~\eqref{eq:L^k_supp}, $L_Q^p[\de_\si'](\si)=0$. This gives the result.
\end{proof}

\ouredit{When the state space is finite, as in the case of stochastic particle systems on finite lattices, then the commutator $C$ is a matrix indexed by the different states. An implication of Lemma~\ref{lem:comm_support} is that there is a re-ordering of the rows/columns that turns $C$ into a banded matrix. Regardless, we can now prove a general result for the asymptotics of the relative entropy rate.}

\begin{theorem}
\label{th:main_result}
\ouredit{Consider $\Dt\in (0,1)$ and $\Po(\si,\si')=e^{L\Dt}\de_{\si'}(\si)$,} $\Pb(\si,\si')$ be an approximation of $\Po$ based on a splitting scheme with $L_1,L_2$ restrictions of the generator $L$ and $\mb$ the stationary measure corresponding to $\Pb$. Then, if the splitting scheme is of order $p$, we define the \textit{bounded} diameter of the state space as $\hat{k}$,
$$
\hat{k}=\min\{\mathrm{diam}(S),p\}=\min\{\max_{\si,\si'}\{d(\si,\si')\},p\}.
$$
Then, if $C(\si,\si')\neq 0$ for at least one pair $\si,\si'\in S$ such that $d(\si,\si')=\hat{k}$, we have that
$$
H(\Pb|\Po)=O(\Dt^{2p-(\hat{k}+1)}).
$$
\end{theorem} 
\begin{proof}
The proof of this theorem is the generalization of the argument given for Theorem~\ref{th:spparks_th}. Picking up from formula~\eqref{eq:proof7}, 

\begin{align}
J(\Dt;\si, \si')=\frac{(C(\si,\si'))^2}{2\Pb(\si,\si')+\Dt^pC(\si,\si')+o(\Dt^p)}\Dt^{2p}+o(\Dt^{2p-\hat{k}}),
\label{eq:F_general_th}
\end{align} 
Our goal is to show that $J(\Dt;\si,\si')=O(\Dt^{2p-\hat{k}})$ for some $(\si,\si')$ and that this is the highest order attainable. Next, let us have $(\si,\si')\in S\times S$ such that $d(\si,\si')=\hat{k}$. Then, from~\eqref{eq:gen_split_power_exp} and~\eqref{eq:L^k_supp}, we have that
\begin{align}
\Pb(\si,\si')=\sum_{k=\hat{k}}^{\infty}\frac{L_{Q}^{k}[\de_{\si'}](\si)}{k!}\Dt^k=O(\Dt^{\hat{k}}),\ \Dt\in (0,1].
\label{eq:order_of_pb}
\end{align} 
Thus from \eqref{eq:F_general_th} and \eqref{eq:order_of_pb}, we can expose the first term of the asymptotic expansion of $F$ as
\begin{align}
J(\Dt;\si, \si')=
	\begin{cases}
		\frac{(C(\si,\si'))^2}{2L_Q^{\hat{k}}[\de_{\si'}](\si)/k!}\Dt^{2p-\hat{k}}+o(\Dt^{2p-\hat{k}}),& \hat{k}<p,\\
		\frac{(C(\si,\si'))^2}{2L_Q^{\hat{k}}[\de_{\si'}](\si)/k!+C(\si,\si')}\Dt^{p}+o(\Dt^{p}),& \hat{k}=p.
	\end{cases}
\end{align}

Next, we need to address the contribution of the rest of the expansion used (see proof of Theorem \ref{th:spparks_th}), that is
\begin{align*}
G(\Dt;\si,\si')=\Pb(\si,\si')\sum_{k=1}^{\infty}\frac{1}{2k+1}\left(\frac{\Pb(\si,\si')-\Po(\si,\si')}{\Pb(\si,\si')+\Po(\si,\si')}\right)^{2k+1}.
\end{align*}
If $\hat{k}<p$, then $G(\Dt;\si,\si')=O(\Dt^{3p-2\hat{k}})$, which are lower order terms given that $\Dt\leq 1$. However, if $\hat{k}=p$, $G(\Dt;\si,\si')=O(\Dt^{p})$ and in fact every term of the series in $G$ is of that order. 

Finally, $H(\Pb|\Po)$ can never have higher order than $p-1$, as that would require $(\si,\si')$ such that $d(\si,\si')>p+1$ and then $C(\si,\si')=0$ (from Lemma \ref{lem:comm_support}).
\end{proof}

\ouredit{The assumption on the commutator in Theorem~\ref{th:main_result} is simple to check for parallel KMC, as we can write down the commutator $C(\si,\si')$ explicitly. For example, for Lie, $C(\si,\si')$ is given by Equation~\eqref{eq:comm_Lie}, so checking the assumption is just a matter of calculation. Additionally, to find the bounded diameter $\hat{k}=\min\{\diam(S),p\}$, it is sufficient to have lower bounds for the diameter, $\diam(S)$, as the order of the local error of the scheme, $p$, will typically be much smaller. Example~\ref{sec:Markov-Chain-Example} shows a case where $p$ is close to $\diam(S)$ and the implications this has for the RER.
}

\subsection{Markov chain example}
\label{sec:Markov-Chain-Example}
In order to illustrate the connectivity-RER relation, we are studying a simple example where we can compute the RER and all related quantities explicitly, either by hand or any symbolic algebra system. All calculations of the RER in this example are not from sampling but by using definition \eqref{eq:rer}.

We study the case of a Markov process with transition rate matrix, $Q$ and $\diam(S)=2$. \ouredit{We consider a positive $\Dt$, $\Dt<1$, and}
$$
Q=\left(
\begin{array}{ccc}
 -3 & 1 & 2 \\
 3 & -4 & 1 \\
 1 & 0 & -1 \\
\end{array}
\right).
$$
Given this, we can calculate the transition probability matrix of the Markov chain as the matrix exponential of $Q$, $\Po(\si,\si')=\exp(\Dt Q)\de_{\si'}(\si)$. Our system has diameter equal to two since $Q_{3,2}=0$ but $Q_{3,1}\cdot Q_{1,2}\neq 0$. We can construct approximations of $\Po$ by splitting $Q$ into \ouredit{components $A,B$ with $Q=A+B$, similarly to how we expressed the generator $L$ as $L_1+L_2$}. One way to do this is
$$
A=\left(
\begin{array}{ccc}
 -3 & 1 & 2 \\
 3 & -4 & 1 \\
 0 & 0 & 0 \\
\end{array}
\right), B=\left(
\begin{array}{ccc}
 0 & 0 & 0 \\
 0 & 0 & 0 \\
 1 & 0 & -1 \\
\end{array}
\right).
$$
Thus, one approximation of $\exp(Q\Dt)$ could be $\exp(A\Dt)\exp(B\Dt)$, which corresponds to the Lie splitting. From Theorem \ref{th:main_result}, since $\diam(S)=p=2$, we expect $H(\Pbl|\Po)=O(\Dt^1)$. This is indeed the case, as,
$$
H(\Pbl|\Po)\simeq 0.124 \Dt-0.0566 \Dt^2+O\left(\Dt^3\right).
$$
The use of $\simeq$ comes from a truncation of the coefficients to three significant digits. We can work similarly with the Strang splitting, now using $\exp(A\Dt/2)\exp(B\Dt)\cdot \exp(A\Dt/2)$ as the approximation to $\Po$. The local order of the Strang splitting is $p=3$, so we expect that $H(\Pbs|\Po)=O(\Dt^{2\cdot 3-3})=O(\Dt^3)$ (see Theorem~\ref{th:main_result}). This can be readily demonstrated by a calculation of the RER, followed by the derivation of its asymptotic expansion:
$$
H(\Pbs|\Po)\simeq 0.0279 \Dt^3+0.000672 \Dt^4+O\left(\Dt^5\right).
$$

\section{Quantifying information loss in transient regimes}
\label{sec:transient}
\ouredit{In this last section, we consider the case where we wish to study the performance of the operator splitting scheme in a transient regime, before convergence to the stationary distribution takes place. Note that in the proofs of Theorems~\ref{th:spparks_th} and~\ref{th:main_result}, we derived the asymptotic expressions of the various quantities without referring to the stationary measure $\mb$. Therefore those results do not depend on the choice of the sampling measure. That is, with the assumptions of Theorem~\ref{th:main_result} and $\nu$ a probability distribution on the state space $S^M$ such that $\nu(\si)>0$ for all states $\si$, then 
\begin{align}
H_{\nu}(\Pb|\Po)=\sum_{\si\in S^M}\nu(\si)\Pb(\si_0,\si_1)\frac{\Pb(\si_0,\si_1)}{\Po(\si_0,\si_1)}=O(\Dt^{2p-\hat{k}}).
\label{eq:general_rer}
\end{align}
}
\ouredit{Therefore, the order of the RER is independent of the sampling measure. As a result, we gain Theorem~\ref{th:general_initial_dist}, an extension of Theorem~\ref{th:main_result} to transient time regimes.}
\begin{theorem}
With the assumptions of Theorem~\ref{th:main_result} for the RER, we have that for any $T>0$:
\begin{align}
\frac{R(Q_{0:T}|P_{0:T})}{T} = \frac{R(\mu_0|\nu_0)}{T}+O(\Dt^{2p-\hat{k}}).
\end{align}
\label{th:general_initial_dist}
\end{theorem}
Theorem~\ref{th:general_initial_dist} is implied by the decomposition of the relative entropy in terms of rates that depend on $\nu_i$ (first discussed in Section~\ref{ssec:info_theory_concepts}). \ouredit{If $M$ is a positive integer, $\Dt$ is the scheme's time-step, and $T=M\Dt$}, then
\begin{align}
R(Q_{0:T}|P_{0:T})&=R(\mu_0|\nu_0)+\sum_{i=1}^{M}H_{\nu_i}(\Pb|\Po).
\label{eq:gen_re_rer_rel_2}
\end{align}
\begin{proof}[Proof of Theorem~\ref{th:general_initial_dist}]
From Equation~\eqref{eq:general_rer} we have that the order of the RER does not depend on the sampling measure $\nu$, as long as $\nu(\si)>0$ for all $\si$. Therefore, $H_{\nu_i}(\Pb|\Po)=O(\Dt^{2p-\hat{k}})$ for $i=1,\ldots,M$. This, combined with Equation~\eqref{eq:gen_re_rer_rel_2}, implies the result. 
\end{proof}

\ouredit{Therefore, our results about the RER are applicable for parallel KMC even for practitioners that are interested in simulating the dynamics in the transient regime.}

\firstRef{
\begin{remark}[Relative Entropy Rate vs.\ path-wise Relative Entropy]
In Section~\ref{ssec:info_theory_concepts}, we saw that, in the stationary regime, we can relate the path-wise relative entropy with the RER via
$$
R(\Qpath|\Ppath)=TH(\Pb|\Po)+R(\Pb|\Po).
$$
In this section, we connected the RER with the RE for transient regimes by using Relation~\eqref{eq:gen_re_rer_rel_2}. Ultimately, those relations motivate the use of the RER as an information criterion in place of the path-wise RE, but there are other advantages too: 
\begin{enumerate}
\item The RER does not depend on the length of the simulated path. Additionally, it can be estimated from a single path, while the path-wise RE requires several. 
\item For large $T$, the RE and RER encapsulate the same amount of information about the similarity of $\Pb$ and $\Po$. 
\end{enumerate}
\end{remark}
}

\section{Conclusions}
\label{sec:Conc}
We introduced the relative entropy rate (RER), i.e.\  path space relative entropy  per unit time,  as a means  to quantify the long-time accuracy of splitting schemes for stochastic dynamics and in particular Parallel KMC algorithms. 
We demonstrated, using   {\em a posterirori} error expansions, the dependence of RER on the following elements: the  local error analysis of the splitting schemes captured by 
the operator commutators; the local error order $p$ and 
the splitting time step $\Dt$, which in the case of Parallel KMC controls the asynchrony between processors; the diameter of the graph associated with the approximated Markov jump process.

Based on this analysis, we showed  that  RER defines  a computable path-space  information criterion that allows to compare,  select and design  different splitting schemes,  taking into account both  error tolerance  (e.g. accuracy of the scheme) and practical concerns such as  asynchrony and  processor communication cost. \ouredit{It is also appropriate to think of the RER as a diagnostic quantity that can be estimated on systems of smaller size and consequently be used to compare schemes and tune parameters without slowing down the target simulation.}

Finally we note that  numerical analysis of   stochastic systems is typically concerned with controlling  the weak error for  observable functions $\phi$,
\begin{align}
\sup_{0\leq n\leq N}\left | \mathbb{E}_{P_{0:T}}[\phi(X(n\Dt))] -\mathbb{E}_{Q_{0:T}}[\phi(X_n)]\right|,
\label{eq:weak_error}
\end{align}
where $X_{n}$ represents the approximate chain and $X(n\Dt)$ the $\Dt$-skeleton chain of the exact process, $T=M\cdot \Dt$. However, our results measure the information loss on path space between 
the approximate chain and  the $\Dt$-skeleton chain of the exact process, using RER. Controlling RER also implies upper bounds for observables at long times, using uncertainty quantification information inequalities developed in~\cite{Dupuis, jie-scalable-info-ineq}. We also showed how those results can be extended to finite-time regimes.

\bibliography{rer_bib}{}

\begin{thebibliography}{10}

\bibitem{Janhke}
Tobias Jahnke and Derya Altıntan.
\newblock Efficient simulation of discrete stochastic reaction systems with a
  splitting method.
\newblock {\em BIT Numerical Mathematics}, 50(4):797--822, 2010.

\bibitem{Arampatzis:2012}
Giorgos Arampatzis, Markos~A. Katsoulakis, Petr Plech\'{a}\v{c}, Michela
  Taufer, and Lifan Xu.
\newblock Hierarchical fractional-step approximations and parallel kinetic
  monte carlo algorithms.
\newblock {\em J. Comput. Phys.}, 231(23):7795--7814, October 2012.

\bibitem{spparks}
Steve Plimpton, Corbett Battaile, Mike Chandross, Liz Holm, Aidan Thompson,
  Veena Tikare, Greg Wagner, E~Webb, X~Zhou, C~Garcia Cardona, et~al.
\newblock Crossing the mesoscale no-man's land via parallel kinetic monte
  carlo.
\newblock {\em Sandia Report SAND2009-6226}, 2009.

\bibitem{Arampatzis:2014}
Giorgos Arampatzis, Markos~A. Katsoulakis, and Petr Plech\'{a}\v{c}.
\newblock Parallelization, processor com munication and error analysis in
  lattice kinetic monte carlo.
\newblock {\em SIAM Journal on Numerical Analysis}, 52(3):1156--1182, 2014.

\bibitem{Petzold}
Andreas Hellander, Michael~J. Lawson, Brian Drawert, and Linda Petzold.
\newblock Local error estimates for adaptive simulation of the
  reaction-diffusion master equation via operator splitting.
\newblock {\em J. Comput. Phys.}, 266:89--100, June 2014.

\bibitem{Engblom}
Stefan Engblom, Lars Ferm, Andreas Hellander, and Per Lötstedt.
\newblock Simulation of stochastic reaction-diffusion processes on unstructured
  meshes.
\newblock {\em SIAM Journal on Scientific Computing}, 31(3):1774--1797, 2009.

\bibitem{Bayati}
Basil~S Bayati.
\newblock Fractional diffusion-reaction stochastic simulations.
\newblock {\em The Journal of chemical physics}, 138(10):104117, 2013.

\bibitem{Talay-Tubaro}
Denis Talay and Luciano Tubaro.
\newblock Expansion of the global error for numerical schemes solving
  stochastic differential equations.
\newblock {\em Stochastic Analysis and Applications}, 8(4):483--509, 1990.

\bibitem{MST10}
Jonathan~C. Mattingly, Andrew~M. Stuart, and Michael Tretyakov.
\newblock Convergence of numerical time-averaging and stationary measures via
  the poisson equation.
\newblock {\em SIAM Journal of Numerical Analysis}, 48:552--577, 2010.

\bibitem{Abdulle}
Assyr Abdulle, Gilles Vilmart, and Konstantinos~C. Zygalakis.
\newblock Long time accuracy of lie--trotter splitting methods for langevin
  dynamics.
\newblock {\em SIAM Journal on Numerical Analysis}, 53(1):1--16, 2015.

\bibitem{Leimkuhler}
Benedict Leimkuhler, Charles Matthews, and Gabriel Stoltz.
\newblock The computation of averages from equilibrium and nonequilibrium
  langevin molecular dynamics.
\newblock {\em IMA Journal of Numerical Analysis}, page dru056, 2015.

\bibitem{irreversibility_KPR}
Markos Katsoulakis, Yannis Pantazis, and Luc Rey-Bellet.
\newblock Measuring the irreversibility of numerical schemes for reversible
  stochastic differential equations.
\newblock {\em ESAIM: Mathematical Modelling and Numerical Analysis},
  48:1351--1379, 9 2014.

\bibitem{sensitivity_PK}
Yannis Pantazis and Markos~A. Katsoulakis.
\newblock A relative entropy rate method for path space sensitivity analysis of
  stationary complex stochastic dynamics.
\newblock {\em The Journal of Chemical Physics}, 138(5), 2013.

\bibitem{Eva-coarse-grain}
Evangelia Kalligiannaki, Markos~A. Katsoulakis, and Petr Plech\'{a}\v{c}.
\newblock Spatial two-level interacting particle simulations and information
  theory--based error quantification.
\newblock {\em SIAM Journal on Scientific Computing}, 36(2):A634--A667, 2014.

\bibitem{Dupuis}
Paul Dupuis, Markos~A. Katsoulakis, Yannis Pantazis, and Petr Plech\'{a}\v{c}.
\newblock Path-space information bounds for uncertainty quantification and
  sensitivity analysis of stochastic dynamics.
\newblock {\em SIAM/ASA Journal on Uncertainty Quantification}, 4(1):80--111,
  2016.

\bibitem{Kipnis-Landim}
Claude Kipnis and Claudio Landim.
\newblock {\em Scaling limits of interacting particle systems}.
\newblock Grundlehren der mathematischen Wissenschaften. Springer, Berlin, New
  York, 1999.
\newblock Appendix 1.

\bibitem{Pazy:1983}
Amnon Pazy.
\newblock {\em Semigroups of Linear Operators and Applications to Partial
  Differential Equations}.
\newblock Springer New York, 1983.

\bibitem{Trotter:1959}
Hale~F. Trotter.
\newblock On the product of semi-groups of operators.
\newblock {\em Proceedings of the American Mathematical Society}, 10(4):pp.
  545--551, 1959.

\bibitem{Cover-Thomas}
Thomas~M. Cover and Joy~A. Thomas.
\newblock {\em Elements of Information Theory}.
\newblock Wiley-Interscience, New York, NY, USA, 1991.

\bibitem{CD:13}
Kamaljit Chowdhary and Paul Dupuis.
\newblock Distinguishing and integrating aleatoric and epistemic variation in
  uncertainty quantification∗.
\newblock {\em ESAIM: M2AN}, 47(3):635--662, 2013.

\bibitem{jie-scalable-info-ineq}
Markos~A.\ Katsoulakis, Luc Rey-Bellet, and Jie Wang.
\newblock Scalable information inequalities for uncertainty quantification,
  2016.
\newblock arXiv:1605.04184.

\bibitem{Akaike1}
Hirotogu Akaike.
\newblock Information theory and an extension of the maximum likelihood
  principle.
\newblock In Emanuel Parzen, Kunio Tanabe, and Genshiro Kitagawa, editors, {\em
  Selected Papers of Hirotugu Akaike}, Springer Series in Statistics, pages
  199--213. Springer New York, 1998.

\bibitem{Bayes1}
Hirotugu Akaike.
\newblock A new look at the bayes procedure.
\newblock In Emanuel Parzen, Kunio Tanabe, and Genshiro Kitagawa, editors, {\em
  Selected Papers of Hirotugu Akaike}, Springer Series in Statistics, pages
  281--287. Springer New York, 1998.

\end{thebibliography}
\bibliographystyle{unsrt}

\appendix
\section{Coefficients of the Relative Entropy Rate for Lie and Strang}
\label{app:coeff-ads-deso}

\ouredit{For the adsorption-desorption example considered in Section~\ref{sec:ising_example} of the main text we need to estimate the highest-order coefficients $A,B$ for Lie and Strang respectively}. \ouredit{To accomplish this, we have} to collect all the coefficients of $\Dt$ and $\Dt^2$ that appear in the expansion of RER in the proof of Theorem~\ref{th:spparks_th}. The result is a summable series for each coefficient. For Lie, we have 
\begin{align}
&A=\mathbb{E}_{\mu_{L}(\si)}\left[\sum_{x,y\in \Lambda}F_L(\si,\si^{x,y})\right]=\sum_{\si}\mu_{L}(\si)\sum_{x,y\in \Lambda}F_L(\si,\si^{x,y}),\label{eq:exact_a}\\
&F_L(\si,\si'):=C_{L}(\si,\si')M_L(\si,\si')-2L^2_{L}[\de_{\si'}](\si)(\mathrm{arctanh}(M_L(\si,\si'))-M_L(\si,\si')),\label{eq:FA}\\
&M_L(\si,\si'):=C_{L}(\si,\si')/(L^{2}_{L}[\de_{\si'}(\si)]+C_{L}(\si,\si')),\nonumber
\end{align}
where we remind the reader that $L^2_{L}$ stands for all the coefficients of $\Dt^2/2$ in the expansion of the Lie splitting and $C_{L}(\si,\si')=1/2[L_1,L_2][\de_{\si'}](\si)$ is the Lie commutator term. Similarly, for the Strang case,
\begin{align}
&B=\mathbb{E}_{\mu_{S}(\si)}\left[\sum_{x,y,z\in \Lambda}F_S(\si,\si^{x,y,z})\right]=\sum_{\si}\mu_{S}(\si)\sum_{x,y,z\in \Lambda}F_S(\si,\si^{x,y,z}),\label{eq:exact_b}\\
&F_S(\si,\si'):=C_S(\si,\si')M_S(\si,\si')-2L^3_{S}[\de_{\si'}](\si)(\mathrm{arctanh}(M_S(\si,\si'))-M_S(\si,\si')),\label{eq:GB}\\
&M_S(\si,\si'):=C_S(\si,\si')/(L^{3}_{S}[\de_{\si'}](\si)+C_S(\si,\si')).\label{eq:MB}
\end{align}
\ouredit{Since both Equation~\eqref{eq:exact_a} and Equation~\eqref{eq:exact_b} are expected values, we can estimate them as ergodic averages.}

\end{document}